\documentclass[reqno, 12pt]{amsart}
\usepackage{array}
\usepackage{amsmath}
\usepackage{amsfonts}
\usepackage{amssymb}

\usepackage{mathrsfs}
\usepackage{enumerate}
\usepackage{amsthm}
\usepackage{verbatim}
\usepackage{amsmath, amscd}
\usepackage{caption}

\usepackage[usenames,dvipsnames]{color}
\usepackage{bm}
\usepackage{xy}
\xyoption{all}
\usepackage{color}
\usepackage{marvosym}
\usepackage{units}
\usepackage{amsbsy}
\usepackage{hyperref}
\usepackage{tikz}
\usepackage{tikz-cd}
\usepackage{marginnote}
\usetikzlibrary{matrix,arrows,backgrounds}

\newtheorem{theorem}{Theorem}[section]

\newtheorem{lemma}[theorem]{Lemma}
\newtheorem{proposition}[theorem]{Proposition}
\newtheorem{corollary}[theorem]{Corollary}

\newtheorem{assumption}[theorem]{Assumption}

\newtheorem{question}[theorem]{Question}

\theoremstyle{definition}
\newtheorem{definition}[theorem]{Definition}
\newtheorem{remark}[theorem]{Remark}

\newtheorem{example}[theorem]{Example}

\newcommand{\op}[1]{\operatorname{#1}}

\newcommand{\newterm}{\textsf}

\newcommand{\dbcoh}[1]{\operatorname{D}^{\operatorname{b}}(\operatorname{coh }#1)}

\newcommand{\dabsfact}[1]{\operatorname{D}^{\operatorname{abs}}[#1]}

\newcommand{\dabs}[1]{\op{D}^{\op{abs}}[#1]}

\newcommand{\E}{\mathcal{E}}

\newcommand{\cone}{\operatorname{Cone}}
\newcommand{\coker}{\operatorname{coker}}
\newcommand{\Hom}{\operatorname{Hom}}
\newcommand{\conv}{\operatorname{Conv}}

\def\N{\op{\mathbb{N}}}
\def\Z{\op{\mathbb{Z}}}
\def\C{\op{\mathbb{C}}}
\def\R{\op{\mathbb{R}}}
\def\Q{\op{\mathbb{Q}}}

\def\O{\op{\mathcal{O}}}
\def\A{\op{\mathbb{A}}}
\def\P{\op{\mathbb{P}}}

\def\tot{\operatorname{tot}}

\def\E{\op{\mathcal{E}}}
\def\Hom{\op{\mathrm{Hom}}}
\newcommand{\Fact}[1]{\operatorname{Fact }(#1)}
\newcommand{\fact}[1]{\operatorname{fact }(#1)}
\def\Gm{\op{\mathbb{G}_m}}
\def\tot{\op{tot}}
\newtheorem{notation}[theorem]{Notation}

\newcommand{\m}{\mathfrak{m}}

\title[Exoflops]{Toric Exoflops and Categorical Resolutions}

\author[Kelly]{Tyler L. Kelly}

\address{
  \begin{tabular}{l}
   Tyler L. Kelly \\
   \hspace{.1in} School of Mathematical Sciences, Queen Mary University of London, \\ \hspace{.1in} 327 Mile End Rd, London E1 4NS, United Kingdom \\
   \hspace{.1in} Email: {\bf t.l.kelly@qmul.ac.uk} \\
  \end{tabular}
}

\author[Malter]{Aimeric Malter}
\address{
  \begin{tabular}{l}
   Aimeric Malter \\
   \hspace{.1in} Beijing Institute of Mathematical Sciences and Applications, A3, Room 4-101, \\
   \hspace{.1in}  No. 544, Hefangkou Village, Huaibei Town, Huairou District, Beijing 101408, China\\
   \hspace{.1in} Email: {\bf aimericmalter@bimsa.cn} \\
  \end{tabular}
}

\numberwithin{equation}{section}
\begin{document}

\begin{abstract}
An exoflop takes a gauged Landau-Ginzburg (LG) model, partially compactifies it, and then performs certain birational transformations on it. When certain criteria hold, this can provide a crepant categorical resolution or equivalence of derived categories associated to the gauged LG models. We provide sufficient criteria for when this provides categorical resolutions for (or equivalences between) certain complete intersections in toric stacks. 
\end{abstract}

\maketitle
\setcounter{tocdepth}{1}
\tableofcontents

\section{Introduction}

Aspinwall proposed the exoflop as a tool for finding categorical resolutions of singular varieties via their corresponding gauged Landau-Ginzburg models  \cite{Aspinwall}. In this work, he computed categorical resolutions of certain singular K3 surfaces in toric varieties to demonstrate their power in this context. In the present paper, we formalize and generalize this usage of exoflops, yielding a dynamic web of related derived categories.

Let $\C$ be an algebraically closed field of characteristic zero. Let $X$ be a smooth variety over $\C$ and $G$ an affine algebraic group that acts on $X$. Take $W$ to be a $G$-invariant section of an invertible $G$-equivariant sheaf $\mathcal{L}$, that is, $W \in \Gamma(X, \mathcal{L})^G$. We call the data $(X, G, W)$ a \newterm{gauged Landau-Ginzburg (LG) model}. There is a derived category called the (matrix) factorization category  $\dabsfact{X,G,W}$ associated to the gauged Landau-Ginzburg model.  To any complete intersection, there is a corresponding gauged Landau-Ginzburg model associated to it \cite{Orlov2, Isik, Shipman, Hirano}. Relating complete intersections to their corresponding gauged Landau-Ginzburg models has had applications for both derived categories and enumerative geometry.

Geometrically, an exoflop can be viewed as having two steps: 
\begin{itemize}
\item[\underline{exo}:] partially compactifying the space $X$ in a gauged LG model while extending the group action $G$ and global function $W$, and 
\item[\underline{flop}:] a birational transformation of the partial compactification of $X$ determined by varying the stability parameter of a prescribed geometric invariant theory quotient.
\end{itemize} 

Exoflops are not a new concept. Over three decades ago, Aspinwall, Greene, and Morrison examined the multiple mirror phenomenon using toric variations of geometric invariant theory \cite{AGMexo}. Here, they studied the secondary fan of $\op{tot} K_{\P^4 / \Z_5^3}$, which is related to the gauged LG model corresponding to the mirror quintic. There, they discovered chambers of the secondary fan that are partial compactifications of line bundles over distinct toric varieties. This testifies to the need of formal results in the context of modern mirror symmetry statements, such as Kontsevich's Homological Mirror Symmetry ~\cite{KontsevichICM}. 

Exoflops were used to show derived equivalences between Calabi-Yau varieties that were BHK mirror to two Calabi-Yau hypersurfaces in the same toric variety \cite{FK19}. Such results were then expanded to understanding Calabi-Yau hypersurfaces in toric varieties \cite{DFK}. Their usage to understand categorical resolutions of Kuznetsov components has been shown in examples in (unpublished) work of Aspinwall, Addington, and Segal that has been outlined in \cite{AAS}, as well as in \cite{FK18}. 

Lastly, they have been studied for toric complete intersections \cite{Aspinwall, Malter}; however, results were only established in examples. These results for complete intersections required involved computations for each example, which motivated finding a new approach.  Using the framework of categorical resolutions avoids these calculations for complete intersections. We thus obtain a more general result for complete intersections in toric Fano varieties. Theorem~\ref{main theorem} below proves sufficient criteria for when a Calabi-Yau complete intersection is derived equivalent to a Calabi-Yau variety with a Batyrev-Borisov mirror. Skip to \textsection\ref{Precise} for precise statements. However, this paper also offers results to readers wanting to relate derived categories of complete intersections in toric varieties, so we first give an intuition to  the results found along the way.

\subsection{How the exoflop works.} Here is a rough overview of how this operation changes the factorization category.  Start with a complete intersection $\mathcal{Z}$ (inside a projective (Fano) toric variety $Y$), and take its corresponding gauged LG model $(X,G,W)$. The gauged LG model is found by taking the vector bundle $\mathcal{E}$, where $\mathcal{Z} = Z(s)$ for some $s \in \Gamma(Y,\mathcal{E})$, writing $[X/G] = \op{tot}\mathcal{E}^\vee$ and $W$ is found by pairing with the section $s$. Partially compactify by finding an open immersion $[X/G] \hookrightarrow [\bar X/\bar G]$ where $W$ can be extended to a $\bar G$-invariant global function $\bar W:\bar X\to \mathbb{A}^1$. When the original complete intersection is not smooth, one can see that the critical locus is not proper. Choosing a good partial compactification can remedy this and yield a crepant categorical resolution (\cite[Theorem 3.7]{FK18}, restated as Theorem~\ref{partial compact CCR} here). This is made explicit in Aspinwall's example, which we treat in \textsection\ref{subsec:Aspinwall}. 

Next, we observe that if one starts with a smooth complete intersection $\mathcal{Z}$, then we have a fully faithful functor $F: \dbcoh{\mathcal{Z}} \to \dabs{\bar X, \bar G, \bar W}$ (Corollary~\ref{cor CCR 1}). As a direct corollary, if $\dabs{\bar X, \bar G, \bar W}$ is Calabi-Yau, then $F$ is an equivalence (Corollary~\ref{cor: CCR equiv}). These corollaries are key observations introduced in this paper. Previously, Orlov proved that if one partially compactifies $X$ while not extending the critical locus of $W$, then the corresponding factorization category is equivalent to the original one \cite[Proposition 1.14]{Orlov04}. This hypothesis manifests as a computationally heavy problem in proving certain containment of ideals to obtain an equivalence (see e.g., \cite[Lemma 5.13]{FK19}, \cite[Lemma 4.5]{Malter}). This became a source of technical difficulty for general results, but our methods sidestep this issue through using crepant categorical resolutions and extend techniques from the hypersurface case in \cite{DFK, FK19} to complete intersections.

Next, one uses variations of geometric invariant theory on the gauged LG model  $(\bar X, \bar G, \bar W)$. From the above viewpoint, this can be optional to find partial compactifications or equivalences; however, after a birational modification one may find a new gauged LG model $(X', G', W')$ that is a more natural LG model (e.g., a toric vector bundle) and proving certain properties about the category can be more straightforward. Varying geometric invariant theory quotients (VGIT) and its ramifications on the factorization category for a gauged Landau-Ginzburg model has been established by Ballard-Favero-Katzarkov and Halpern-Leistner \cite{BFK, HL}. By first partially compactifying, one obtains more relations to other factorization categories of gauged Landau-Ginzburg models and more geometric invariant theory problems. A key part to using this technique is finding the right context and partial compactification. We provide a result in convex geometry that in turn allows us to identify ways to ensure our partial compactification is a GIT quotient and allows us to use techniques from VGIT (Lemma~\ref{ExtendRegtriang} / Corollary~\ref{reg triang}).

Lastly, we note when one does this carefully, one can sometimes find that the new gauged Landau-Ginzburg model corresponds to a different complete intersection $\mathcal{Z}'$, and one can establish relations between $\mathcal{Z}$ and $\mathcal{Z}'$. Roughly, the sequence of relations is the following:
\begin{equation}\begin{tikzcd}
\mathcal{Z} \arrow[r, "LG/CI"] & (X, G, W) \arrow[l] \arrow[d, shift left, "exo"]\\
& (\bar X, \bar G, \bar W) \arrow[u, shift left] \arrow[d, "flop"] \\
\mathcal{Z}' \arrow[r, "LG/CI"] & (X', G', W') \arrow[l] \arrow[u]
\end{tikzcd}
\end{equation}

Aspinwall's original paper aimed to establish the use of exoflops for K3 surfaces in toric varieties. 
The results in Sections~\ref{sec: exo} and~\ref{sec: flop} work more generally. We summarise the most general versions of results \textsection\ref{subsec:summary}; however, there is no uniform theorem outside the CICY case. The cleanest results hold for Calabi-Yau complete intersections whose defining polynomials are generic with respect to special linear systems corresponding to completely split reflexive Gorenstein cones.

\subsection{Precise results}\label{Precise}

Let $M$ and $N$ be dual lattices. Suppose $X_{\Sigma}$ is a toric projective $\Q$-Fano variety with fan $\Sigma$ in $N_{\R}$ and assume the fan $\Sigma$ is  simplicial. We recall that each ray $\rho \in \Sigma(1)$ corresponds to a torus-invariant Weil divisor $D_\rho$ and $\sum_{\rho\in\Sigma(1)} D_\rho = -K_{X_{\Sigma}}$. We can make a Calabi-Yau complete intersection in $X_{\Sigma}$ as follows. Take $D_1, \dots, D_r$ to be torus-invariant Weil divisors that partition $-K_{X_{\Sigma}}$. That is, let $R_1 \sqcup \dots \sqcup R_r = \Sigma(1)$ and take
$$
D_i = \sum_{\rho \in R_i} D_\rho.
$$

Note that there are polytopes $P_{D_i}$ whose lattice points correspond to a basis of the vector space of global sections $\Gamma(X_{\Sigma}, \mathcal{O}_{X_{\Sigma}}(D_i))$. Take global sections $f_i\in\Gamma(X_{\Sigma}, \mathcal{O}_{X_{\Sigma}}(D_i))$ and write it as $f_i = \sum_{m \in P_{D_i}} c_m x^m$. Then consider the complete intersection $\mathcal{Z} = Z(f_1, \dots, f_r) \subseteq \mathcal{X}_{\Sigma}$ where $\mathcal{X}_{\Sigma}$ is the toric stack associated to the fan $\Sigma$.

Consider the toric fan $\Sigma_{-D_1, \dots, -D_r}$ where $X_{\Sigma_{-D_1, \dots, -D_r}}$ is the total space $\op{tot}(\oplus_{i=1}^r \mathcal{O}_{X_{\Sigma}}(-D_i))$ (for an explicit construction, see \textsection\ref{subsec:toric vector bundles}).  The cone $\sigma= |\Sigma_{-D_1, \dots, -D_r}|$ that supports the fan $\Sigma_{-D_1, \dots, -D_r}$ is strictly convex. Its dual cone can be written as 
$$
\sigma^\vee = \operatorname{Cone}(P_{D_1} * \cdots *P_{D_r}) \subseteq N_{\R} \times \R^r,
$$
where $P_{D_1} * \cdots *P_{D_r}$ is a Cayley polytope given by taking the convex hull of $P_{D_i} + e_i$, where $e_i$ are the standard elementary basis vectors for $\R^r$. We can make a subpolytope by taking $$\Xi_i := \{ m \in P_{D_i} \ | \ \text{ the coefficient $c_m$ for the monomial $x^m$ in $f_i$ is nonzero}\}.$$
Take the convex hull  $P_i = \op{Conv}(\Xi_i)$. Note $P_i \subseteq P_{D_i}+e_i$, and to get a proper partial compactification in the exoflop we need this containment to be \emph{strict}. This means we will be considering a \emph{special} linear system of  $D_i$.  Construct the cone $\sigma_W := \operatorname{Cone}( P_1* \cdots *P_{r})$ and  
consider its the dual cone, $\sigma_W^\vee$. The choice of global sections $f_i$ completely determines this cone, and under appropriate circumstances the choice we make allows us to write $\sigma_W^\vee$ as support of a toric variety that is itself a rank $r$ vector bundle $\bigoplus_{i=1}^r\O_{X_{\Sigma'}}(D_i)$ over another complete toric stack $\mathcal{X}_{\Sigma'}$. In this case, we can construct a new Calabi-Yau complete intersection $\mathcal{Z}'\subset\mathcal{X}_{\Sigma'}$. In fact, one notes that the construction of the complete intersection remains valid for any cone $\sigma'$ lying between $\sigma_W$ and $\sigma^\vee$ whose dual corresponds to the support of an appropriate rank $r$ vector bundle, which is the content of the key Assumption~\ref{conical assumption}. We also assume that the Calabi-Yau orbifolds $ \mathcal Z, \mathcal Z'$ are positive dimensional.

\begin{theorem}[Corollary \ref{2 CICYs}] Suppose Assumption~\ref{conical assumption} holds and take $\mathcal Z, \mathcal Z'$ defined above.
\begin{itemize}
\item[(i)] If $\mathcal Z'$ is smooth, then we have a crepant categorical resolution 
\begin{align*}
F&:\dbcoh{\mathcal Z'} \to \dbcoh{\mathcal Z},   \\
G&:\op{Perf} \mathcal Z \to \dbcoh{\mathcal Z'}.
\end{align*}
\item[(ii)] If both $\mathcal Z$ and $\mathcal Z'$ are smooth, then they are derived equivalent.
\end{itemize}
\end{theorem}

We then prove the following result that gives combinatorial sufficient criteria for obtaining a (geometrically realizable) crepant categorical resolution for the complete intersection $\mathcal Z$.

\begin{theorem}[Corollary~\ref{Cor:csrG gives cat resn}]\label{main theorem}
If the cone $\sigma_W$ and its dual are completely split reflexive Gorenstein cones, and the coefficients $c_m$ are generic, then there is an explicitly constructed Calabi-Yau complete intersection $\mathcal{Z}'$ in a toric Fano variety and a crepant categorical resolution of $\mathcal{Z}$ given by two functors 
\begin{align*}
F&:\dbcoh{\mathcal Z'} \to \dbcoh{\mathcal Z},   \\
G&:\op{Perf} \mathcal Z \to \dbcoh{\mathcal Z'}.
\end{align*}
Moreover, if $\mathcal{Z}$ is smooth, then it is a derived equivalence.
\end{theorem}

The Calabi-Yau complete intersections $\mathcal{Z'}$ used in the above theorem have a special place in mirror symmetry. They are precisely those that have mirrors that can be constructed using the Batyrev-Borisov mirror construction \cite{BB2, BB1}. This provides a way to take a smooth Calabi-Yau complete intersection $\mathcal Z$ that fits into this criterion and find its mirror $A$-model, as its mirror $A$-model should be the same as that of $\mathcal Z$ by Theorem~\ref{main theorem}, in light of Kontsevich's Homological Mirror Symmetry Conjecture \cite{KontsevichICM}. We end the paper with several examples\textemdash some relating to this mirror symmetry viewpoint\textemdash with a look to generalizing mirror constructions involving special linear systems defining Calabi-Yau complete intersections in toric varieties.

\subsection*{Acknowledgments}
We thank David Favero heartily for discussions and previous work on related topics. We also thank Nick Addington, Alessandro Chiodo, Cyril Closset, Luigi Martinelli, and Ed Segal for discussions relating to this project. We also thank the referee for their comments that improved the paper. The first author acknowledges support from EPSRC Grant EP/S03062X/1 and a UK Research and Innovation Future Leaders Fellowship MR/T01783X/1,2.  They also thank the Fondation Sciences Mathématiques de Paris for support and the Institut de Math\'ematiques de Jussieu for their hospitality where portions of this paper were written. The second author was supported by the EPSRC grants EP/L016516/1, Postdoctoral Fellowships for Research in Japan (Short-term (PE), Fellowship ID: PE23724), the Beijing Natural Science Foundation International Scientist Program IS25013 and the Beijing Postdoctoral Research Foundation.

\section{Background on factorization categories}

In this section we outline relevant details for the derived category associated to a gauged Landau-Ginzburg model, the factorization category. We aim to suppress technicalities that may otherwise distract from the main narrative of the paper, but provide references.

\subsection{Gauged Landau-Ginzburg models and their factorizations}

Let $\C$ be an algebraically closed field of characteristic zero. Let $X$ be a smooth variety over $\C$ and $G$ an affine algebraic group that acts on $X$. Take $W$ to be a $G$-invariant section of an invertible $G$-equivariant sheaf $\mathcal{L}$, that is, $W \in \Gamma(X, \mathcal{L})^G$. We call the data $(X, G, W)$ a \newterm{gauged Landau-Ginzburg (LG) model}. To a gauged LG model, there is the absolute derived category $\dabsfact{X,G,W}$ associated to it, which is the analogue of a variety's bounded derived category of coherent sheaves.

We roughly outline its definition as follows. Good resources include \cite{BFKMF, Hirano}. 

\begin{definition}
A \newterm{factorization} is the data $\E = (\E_{0}, \E_1, \phi^{\E}_0, \phi_1^{\E})$ where $\E_0$ and $\E_1$ are $G$-equivariant quasi-coherent sheaves and 
$$
\E_0 \stackrel{\phi_0^{\E}}{\longrightarrow} \E_1 \stackrel{\phi_1^{\E}}{\longrightarrow} \E_0 \otimes_{\O_X} \mathcal{L}
$$
are morphisms such that $\phi^{\E}_1 \circ \phi^{\E}_0 = w$ and $(\phi_0^{\E} \otimes \mathcal{L}) \circ \phi_{1}^{\E} = w$.
\end{definition}
For two factorizations $\E$ and $\mathcal{F}$, there is a complex $\Hom(\E, \mathcal{F})$ of morphisms from $\E$ to $\mathcal{F}$ defined as follows. We have the graded vector space
$$
\Hom(\E, \mathcal{F})^\bullet := \bigoplus_{n\in \Z} \Hom(\E, \mathcal{F})^n
$$
with differential $d^i:  \Hom(\E, \mathcal{F})^i \to  \Hom(\E, \mathcal{F})^{i+1}$ given by $d^i(f) = \phi^{\mathcal{F}}_{\star + i} \circ f - (-1)^i f \circ \phi^{\E}_{\star}$
where
\begin{equation*}\begin{aligned}
\Hom(\E, \mathcal{F})^{2m} &:= \Hom(\E_1, \mathcal{F}_1\otimes \mathcal{L}^{\otimes m}) \oplus \Hom(\E_0, \mathcal{F}_0\otimes \mathcal{L}^{\otimes m}) \\
\Hom(\E, \mathcal{F})^{2m+1} &:= \Hom(\E_1, \mathcal{F}_0\otimes \mathcal{L}^{\otimes m}) \oplus \Hom(\E_0, \mathcal{F}_1\otimes \mathcal{L}^{\otimes m+1}) \\
\end{aligned}\end{equation*}
This yields a dg category $\Fact{X,G,W}$. Denote by $\fact{X,G,W}$ the full dg-subcategory of this dg category whose components are coherent. 

This category has a subcategory of acyclic complexes. Given $\Fact{X,G,W}$, consider the subcategory $Z^0\Fact{X,G,W}$ with the same objects but only degree zero morphisms. Given a complex of objects in $Z^0\Fact{X,G,W}$, one can construct a new object $\mathcal{T} \in \Fact{X,G,W}$ in a natural way by taking direct sums and arranging the morphisms in a natural way (see (2.1) of \cite{FK18} for details). Take $\operatorname{Acyc}(X,G,W)$ to be the full subcategory of $\Fact{X,G,W}$ consisting of all totalizations of bounded exact complexes in $Z^0\Fact{X, G, W}$ and let $\operatorname{acyc}(X,G,W) = \operatorname{Acyc}(X,G,W) \cap \fact{X,G,W}$. 
\begin{definition}
The \newterm{absolute derived category} $\dabs{X,G,W}$ is the idempotent completion of the Verdier quotient of $\fact{X,G,W}$ by $\operatorname{acyc}(X,G,W)$. 
\end{definition}

The absolute derived category $\dabs{X,G,W}$ can be thought of as the derived category of the gauged LG model $(X,G,W)$. To justify this claim, we must introduce some context and notation. 

\begin{notation}\label{geometric context notation}
Let $Y$ be a smooth quasi-projective variety with a $G$-action. Suppose that $s$ is a regular section of a $G$-equivariant vector bundle $\E$ on $Y$ with vanishing locus $Z := Z(s)$. Let $\Gm$ act on the total space $\tot \E^\vee$ of the dual bundle to $\E$  by fiberwise dilation (the so-called \newterm{$R$-charge}) and consider the pairing $W=\langle -,s\rangle$ as a section of $\O_{\tot \E^\vee}(\chi)$ where $\chi$ is the projection character. 
\end{notation}

We have the following theorem, which has appeared in various forms due to Orlov \cite{Orlov2}, Isik \cite{Isik}, Shipman \cite{Shipman}, and, in its most general form, Hirano \cite{Hirano}.

\begin{theorem}[Proposition 4.8 of \cite{Hirano}]\label{Orlovs thm}
There exist an equivalence of categories
$$
\Omega: \dbcoh{[Z/G]} \stackrel{\sim}{\longrightarrow} \dabs{\tot \E^\vee, G \times \Gm, W}.
$$
\end{theorem}

\begin{remark}
This theorem is also often viewed as (a variant of) Kn\"orrer periodicity or Orlov's theorem. In fact, one needs to use variations of geometric invariant theory from this theorem to recover Orlov's theorem.
\end{remark}

\subsection{Partial compactifications yielding categorical resolutions}\label{ptl compactifications abstract}

In effect, this subsection gives a mathematical introduction to the `exo' part of the exoflop, following \cite{FK18}. We first recall a categorical resolution of singularities, and provide a sufficient criteria for their existence using gauged LG models. We use the definition stated in loc. cit., but the one given by Kuznetsov \cite[Definition 3.2]{Kuz08} could be used throughout if the reader prefers.  Let $Z$ be a variety with a $G$-action and $\mathcal{D}$ an admissible subcategory of $\dbcoh{[Z/G]}$. We denote by $\mathcal{D}^{\op{perf}}$ the full subcategory of $\mathcal{D}$ consisting of $G$-equivariant perfect complexes on $Z$. 

\begin{definition}[Definition 3.1 of \cite{FK18}]\label{def: CCR}
Let $\tilde{\mathcal{D}}$ be the homotopy category of a homologically smooth and proper pretriangulated dg category. A pair of exact functors
\begin{align*}
F&:\tilde{\mathcal{D}} \to \mathcal{D}\\
G&:\mathcal{D}^{\op{perf}} \to \tilde{\mathcal{D}}
\end{align*}
is a \newterm{categorical resolution of singularities} if $G$ is left adjoint to $F$ and the natural morphism of functors $\op{Id}_{\mathcal{D}^{\op{perf}}} \to FG$ is an isomorphism. We say the categorical resolution is \newterm{crepant} if $G$ is right adjoint to $F$. 
\end{definition}

We first discuss step (exo) and its interactions with the absolute derived category. A good resource for this can be found in \cite[Section 3]{FK18}. Consider a variety $U$ equipped with an action by a linearly reductive group $G$, a character $\chi$ of $G$ and a section $W$ of $\O_U(\chi)$. Let 
$$
i: V \longrightarrow U
$$
be a $G$-equivariant open immersion. 
\begin{definition}[Definition 3.3 of \cite{FK18}]
Let $\dabs{V,G,W}_{\op{rel} U}$ denote the full subcategory of $\dabs{V, G, W}$ consisting of factorizations $\E$ where the closure of the support of $\E$ in $U$ does not intersect $U\setminus V$. 
\end{definition}
We then have the following functors
\begin{equation}\begin{aligned}\label{adjoint inclusions}
i_*: \dabs{V,G,W}_{\op{rel} U} &\longrightarrow \dabs{U,G,W}; \\ 
i^*:  \dabs{U,G,W} &\longrightarrow\dabs{V,G,W};
\end{aligned}\end{equation}
where $i_*$ is both left and right adjoint to $i^*$. We can use this to build crepant categorical resolutions, after giving a geometric context to the categories $\dabs{V,G,W}_{\op{rel} U}$ and $\dabs{V,G,W}$. We do this via the following variant of Theorem~\ref{Orlovs thm}. Recall Notation~\ref{geometric context notation}.
\begin{lemma}[Lemma 3.6 of \cite{FK18}]\label{zero section lemma}
Assume $Y$ admits a $G$-ample line bundle. The equivalence of categories 
$$
\Omega: \dbcoh{[Z/G]} \to \dabs{\tot \E^\vee, G \times \Gm, W}
$$
restricts to an equivalence between the full subcategory of perfect objects $\op{Perf} [Z/G]$ and the full subcategory of $\dabs{\tot \E^\vee, G \times \Gm, W}$ with objects supported on the zero section of $\E^\vee$. 
\end{lemma}
Note that if the zero section is already proper, then the image of the subcategory of perfect objects will not intersect with the partial compactification. Thus, to find a crepant categorical resolution, we find an appropriate partial compactification of $\tot \E^\vee$. 

\section{Partial compactifications for toric gauged LG models}\label{sec: exo}

In this section, we contextualise the ideas in \textsection\ref{ptl compactifications abstract} to a toric setting. 
Fix a lattice $M$ of rank $d$ with dual lattice $N$, equipped with the pairing\[
\langle -,-\rangle: M\times N\rightarrow \Z.
\]
\noindent Extend this pairing $\R$-linearly to $M_{\R} := M \otimes_{\Z}\R$ and $N_{\R} := N \otimes_{\Z}\R$. Often we will consider the spaces $M_{\R}\times \R^r$ and $N_{\R}\times \R^r$ for some $r\in \Z$ and extend the inner product using the Euclidean inner product.

\subsection{Cox stacks}
Let $\Sigma$ be a fan in $N_{\R}$. We define a quotient stack $\mathcal{X}_{\Sigma}$ associated to the fan $\Sigma$ following the Cox construction as follows (see, e.g., \cite[Section 5.1]{CLS} for the standard construction). Let $\nu = \{u_{\rho} \ | \ \rho \in \Sigma(1)\} \subseteq N$, where $u_\rho$ is the primitive lattice generator of the ray $\rho \in \Sigma(1)$. Consider the vector space $\R^{\Sigma(1)}$ with elementary $\Z$-basis vectors $e_\rho$ for each $\rho \in \Sigma(1)$. We construct a new fan
\begin{equation}
\op{Cox}(\Sigma) := \{ \op{Cone}(e_\rho \ | \ \rho \in \sigma) \ | \ \sigma \in \Sigma\}
\end{equation}
This fan is a subfan of the standard fan for $\mathbb{A}^{| \Sigma(1)|}$, hence its associated toric variety is an open subset of affine space. 

\begin{definition}
We call $U_\Sigma:= X_{\op{Cox}(\Sigma)}$ the \newterm{Cox open set} associated to $\Sigma$.
\end{definition}
There is a canonical group acting on $U_\Sigma$. Consider the right exact sequence 
\begin{equation}
M \stackrel{f_{\nu}}{\longrightarrow} \Z^{\nu} \stackrel{\pi}{\longrightarrow} \op{coker}(f_\nu) \longrightarrow 0
\end{equation}
where $f_{\nu}(m) := \sum_{\rho \in \Sigma(1)} \langle u_\rho, m\rangle e_\rho$. Applying the functor $\Hom(-,\mathbb{G}_m)$ to this sequence yields the left exact sequence\[
1\rightarrow \Hom(\coker(f_\nu),\mathbb{G}_m)\xrightarrow{\widehat{\pi}}\mathbb{G}_m^{\nu} \xrightarrow{\widehat{f_\nu}}\mathbb{G}_m^d.
\]
Define \[
S_{\Sigma}:=\Hom(\coker(f_\nu),\mathbb{G}_m).
\]
We note that $S_{\Sigma}$ acts on the open set $U_\Sigma$ constructed above.

\begin{definition}
    \label{Def:Coxstack} 
    Define the \newterm{Cox stack} associated to $\Sigma$ to be \[
   \mathcal{X}_\Sigma:=[U_\Sigma/S_{\Sigma}].
   \]
\end{definition}

\subsection{Toric vector bundles}\label{subsec:toric vector bundles}

Let $\Sigma$ be a fan in $N_{\R}$. For each ray $\rho\in\Sigma(1)$ with primitive generator $u_\rho$, denote by $D_\rho$ the torus-invariant divisor associated to it. Recall that any torus-invariant Weil divisor $D$ on $X_\Sigma$ can be written as a linear combination
\[
D=\sum_{\rho\in \Sigma(1)}a_\rho D_\rho,
\]
for some $a_\rho\in \Z$. Consider now a collection of $r$ such divisors, $D_i=\sum_{\rho\in\Sigma(1)}a_{i\rho}D_\rho$, with $a_{i\rho}\in\Z$.

Write $e_1, \dots, e_r$ for the elementary $\Z$-basis for $\R^r$, and we denote by $e_1^*, \dots, e_r^*$ its dual basis.

For any cone $\sigma\in\Sigma$, define the cone
\[
\sigma_{D_1,\dots,D_r}:=\cone(\{u_\rho-a_{1\rho}e_1-\dots-a_{r\rho}e_r\vert \rho\in\sigma(1)\}\cup\{e_i\mid i\in\{1,\dots,r\})\subseteq N_{\R}\times \R^r.
\]
By considering the collection $\{\sigma_{D_1,\dots,D_r}\mid \sigma\in\Sigma\}$ of these cones, together with their proper faces, we form the fan $\Sigma_{D_1,\dots,D_r}\subseteq N_{\R}\times \R^r$. Note that the rays of $\Sigma_{D_1,\dots,D_r}$ are given by 
$$
\bar\rho := \op{Cone}( u_\rho-a_{1\rho}e_1-\dots-a_{r\rho}e_r) \text{ for all $\rho \in \Sigma(1)$}
$$
and $\epsilon_i := \op{Cone}(e_i)$ for $i \in \{1, \dots, r\}$. 
The toric variety associated to the fan $\Sigma_{D_1,\dots,D_r}$ is the total space of the vector bundle $\bigoplus_{i=1}^r \O(D_i)$ over $X_\Sigma$ (see, e.g., Proposition 7.3.1 of \cite{CLS}). The analogous proposition for toric stacks is the following. 

\begin{proposition}[Proposition 5.16 in \cite{FK14}]
    \label{Prop:FK14Prop5.16}
    Let $D_1,\dots,D_r$ be torus-invariant Weil divisors on $X_\Sigma$ defined as above. There is an isomorphism of stacks\[
    \mathcal{X}_{\Sigma_{D_1,\dots,D_r}}\simeq \tot\left(\bigoplus_{i=1}^r\O_{\mathcal{X}_\Sigma}(D_i)\right).
    \]
\end{proposition}

Recall that the set of global functions on a toric variety $\Sigma$ in $N_{\R}$ can be computed as follows. Take the support $\sigma = |\Sigma|$ and compute the dual cone $\sigma^\vee \subseteq M_{\R}$. Then for any $m \in \sigma^\vee\cap M$, we obtain a monomial
\begin{equation}\label{toric monomial}
x^m := \prod_{\rho \in \Sigma(1)} x_\rho^{\langle m, u_\rho\rangle}
\end{equation}
that is a global function on $X_\Sigma$. Any algebraic map $X_\Sigma \to \C$ can be written as a linear combination of $x^m$ for some $m \in \sigma^\vee\cap M$. While we will often consider $\Sigma$ to be complete, hence only have constant global functions, the supports of the fans $\Sigma_{D_1,\dots,D_r}$ will be strictly convex cones, which will have have maximal dimension dual cones.

We now restrict to $\Sigma$ being a complete fan in $N_{\R}$. The polytope associated to the torus-invariant Weil divisor $D_i= \sum_{\rho \in \Sigma(1)} a_{\rho i} D_\rho$  is given by 
$$
P_{D_i} = \{ m \in M_{\R} \ | \ \langle m, u_\rho\rangle \ge -a_{\rho i} \text{ for all } \rho \in \Sigma(1)\}.
$$
Note that $m \in P_{D_i}$ if and only if $\langle m + e_i^*, u_\rho + a_{\rho i}e_i\rangle \ge 0$ for all $\rho \in \Sigma(1)$. Thus
$m \in P_{D_i}$ if and only if $m + e_i^* \in |\Sigma_{-D_1,\dots,-D_r}|^\vee$. Consider the (codimension $r$) hyperplane $H_i := \{ (m, \delta_{1i}, \dots, \delta_{ri}) \ | \ m \in M_{\R}\}$.\footnote{Equivalently,  the hyperplane can be defined as $H_i = \{ (m, b_1, \dots, b_r) \ | \ b_i =1, b_j = 0 \text{ for all $j \ne i$}\}$.} 
Thus we get that the global sections of each divisor
$$
f_i = \sum_{m \in P_{D_i} \cap M} c_m \prod_{\rho \in \Sigma(1)} x_\rho^{\langle m, u_\rho\rangle+a_{\rho i}} \in \Gamma(X_{\Sigma}, \O(D_i))  $$
correspond to the global functions on $\op{tot} \bigoplus_{i=1}^r \O(-D_i)$
\begin{equation}\begin{aligned} s_i &= \sum_{(m, b_1, \dots, b_r) \in H_i \cap (M\times \Z^r)} c_m \prod_{ \bar\rho \in \Sigma_{-D_1,\dots,-D_r}(1)} x_{\bar\rho}^{\langle m+ \sum_{i=1}^r b_i e_i, u_{\bar\rho}\rangle} \\
	&= u_i \sum_{(m, b_1, \dots, b_r) \in H_i\cap (M \times \Z^r)} c_m \prod_{ \rho \in \Sigma(1)} x_{\bar\rho}^{\langle m, u_\rho\rangle+a_{\rho i}} ,
\end{aligned}\end{equation}
where we denote the coordinate associated to the ray $\epsilon_i$ by $u_i$ and that corresponding to $\bar\rho$ by $x_\rho$.  By abuse of notation, we will write $s_i = u_if_i$, as it is equivalent if one conflates $x_\rho$ with $x_{\bar\rho}$. Write $$[\mathcal{Z}/G] = [Z(f_1, \dots, f_r) / S_{\Sigma}] \subseteq [U_\Sigma / S_{\Sigma}]$$ and $W = \sum_{i=1}^r u_if_i$. We recall the $R$-charge action of $\Gm$ acting on a vector bundle by fiberwise dilation (see Notation~\ref{geometric context notation}) and consider the projection character 
$$
\chi: S_{\Sigma} \times \Gm \to \Gm.
$$
Note that 
$$
W \in \Gamma(U_{\Sigma_{-D_1,\dots,-D_r}}, \O_{U_{\Sigma_{-D_1,\dots,-D_r}}}(\chi))^{S_{\Sigma_{-D_1,\dots,-D_r}} \times \Gm}.
$$
\noindent We remark this is equivalent to $W$ being semi-invariant with respect to the character $\chi$ \cite[Definition 4.3]{FK19}.

\begin{corollary}\label{cor: Hirano}
There exists an equivalence of categories
$$
\Omega: \dbcoh{[\mathcal{Z}/G]} \stackrel{\sim} \longrightarrow \dabs{ U_{\Sigma_{-D_1,\dots,-D_r}}, S_{\Sigma_{-D_1,\dots,-D_r}(1)} \times \Gm, W},
$$
where the $\Gm$ acts with weights 0 on the coordinates $x_\rho$ and 1 on the $u_i$.
\end{corollary}
\begin{proof}
Follows directly from Theorem~\ref{Orlovs thm}.
\end{proof}
In light of the above corollary, the gauged Landau-Ginzburg model associated to the complete intersection $[\mathcal{Z}/G]$ is \[
(U_{\Sigma_{-D_1,\dots,-D_r}}, S_{\Sigma_{-D_1,\dots,-D_r}(1)} \times \Gm, W).
\]
\begin{remark}\label{simplicial replacement}
 One can equip any simplicial fan $\tilde{\Sigma}$ with $\tilde{\Sigma}(1)=\Sigma(1)$ with the same superpotential. This fact becomes important when comparing factorization categories associated to different toric gauged LG models related by exoflops.
 \end{remark}

\subsection{Partial compactifications and crepant categorical resolutions}

We continue with the setup of \textsection~\ref{subsec:toric vector bundles} and keep the notations $\Sigma, D_i, W, H_i, f_i, s_i$ as above. In light of Remark~\ref{simplicial replacement}, we assume $\Sigma$ is simplicial and thus so is $\Sigma_{-D_1,\dots,-D_r}$. Given the global function $W: \op{tot} \bigoplus_{i=1}^r \O(-D_i) \to \mathbb{A}^1$, we can define
$$
\Xi_{i, W} := \{ \bar m = (m, \delta_{1i}, \dots, \delta_{ri}) \in H_i \ | \ c_m \ne 0\}.
$$
Note $\Xi_{i,W} \subseteq P_{D_i}+e_i^*$. Take $\Xi_W = \bigcup_{i=1} \Xi_{i,W}$. Note that $\sigma_W := \op{Cone}(\Xi_W) \subseteq |\Sigma_{-D_1,\dots,-D_r}|^\vee$, hence $\sigma_W^\vee \supseteq |\Sigma_{-D_1,\dots,-D_r}|$. 

Take a (strictly convex) rational polyhedral cone $\sigma'$ so that $ |\Sigma_{-D_1,\dots,-D_r}| \subseteq \sigma' \subseteq \sigma_{W}^\vee$. We have the following result.

\begin{lemma}\label{ptl compactification}
There exists a simplicial fan $\Psi$ with support $\sigma'$ so that $\Sigma_{-D_1, \dots, -D_r}$ is a subfan of $\Psi$.
\end{lemma}
\noindent This lemma is a corollary of Lemma \ref{ExtendRegtriang} proven in  Subsection~\ref{subsec:convex geometry} using convex geometry, so we postpone its proof. 

Using Corollary 4.23 of \cite{FK18}, we have a stack isomorphism $$\varphi: [U_{\Sigma_{-D_1, \dots, -D_r}} \times \Gm^{\Psi(1) \setminus \Sigma_{-D_1, \dots, -D_r}(1)} / S_{\Psi}] \stackrel{\sim}{\longrightarrow} [U_{\Sigma_{-D_1, \dots, -D_r}} / S_{\Sigma_{-D_1, \dots, -D_r}}],$$ which induces an equivalence of categories on the associated absolute derived categories.
 We then consider the $S_\Psi$-equivariant open immersion 
\begin{equation}\label{the immersion}
i: U_{\Sigma_{-D_1, \dots, -D_r}} \times \Gm^{\Psi(1) \setminus \Sigma_{-D_1, \dots, -D_r}(1)} \hookrightarrow U_{\Psi}.
\end{equation}
As shorthand, we write $U_{\Sigma, D_i}' := U_{\Sigma_{-D_1, \dots, -D_r}} \times \Gm^{\Psi(1) \setminus \Sigma_{-D_1, \dots, -D_r}(1)}$. 
Note that the function $W: \op{tot} \bigoplus_{i=1}^r \O(-D_i)\to \mathbb{A}^1$ can be viewed as an $S_{\Sigma_{-D_1, \dots, -D_r}}$-invariant function $W: U_{\Sigma_{-D_1, \dots, -D_r}} \to \mathbb{A}^1$. Extend the $R$-charge trivially to the new variables. 
Since $W$ is a linear combination of toric monomials in the dual cone to $|\Psi|$, the function is then extended to
\begin{equation}\label{extended potential}
\bar W \in \Gamma( U_{\Psi}, \O_{U_\Psi}(\chi))^{S_{\Psi} \times \Gm}
\end{equation}
where $\chi:S_{\Psi} \times \Gm \to \Gm$ is again the projection. Indeed, the extension $\bar W$ is a section of $\O_{U_\Psi}(\chi)$ because each toric monomial $x^{\bar m}$, once written in coordinates, is of the form $u_i \cdot \prod_{\bar \rho \in \Psi(1) \setminus \{\epsilon_i\} } x_{\bar \rho}^{\langle \bar m, u_{\bar \rho}\rangle}$ for some $i$, where the $u_i$ are the bundle coordinates as in \textsection~\ref{subsec:toric vector bundles}. The following is essentially an alternative stating of Theorem 3.7 of \cite{FK18} in the toric setting.

\begin{theorem}\label{partial compact CCR}
Suppose $\dabs{U_{\Psi}, G_{\Psi} \times \Gm, \bar W}$ is homologically smooth and proper. Then we have the following crepant categorical resolution
\begin{equation}\begin{aligned}\label{adjoint inclusions 2}
i_* \circ \varphi^* \circ\Omega: \op{Perf} [\mathcal Z/G] &\longrightarrow \dabs{U_{\Psi}, G_{\Psi} \times \Gm, \bar W}; \\ 
\Omega^{-1}\circ \varphi_* \circ i^*:  \dabs{U_{\Psi}, G_{\Psi} \times \Gm, \bar W} &\longrightarrow \dbcoh{[\mathcal Z/G]};
\end{aligned}\end{equation}
\end{theorem}
\begin{proof} First, we note that using Lemma~\ref{zero section lemma}, the image under $\varphi_* \circ \Omega$ of $\op{Perf} [\mathcal Z/G] $ is supported on the hyperplane $Z(u_1, \dots, u_r)\subseteq U_{\Sigma, D_i}' $. Note that the partial compactification given by the open immersion $i$ given in ~\eqref{the immersion} does not intersect the hyperplane, thus the image under $i_*$ will be in the subcategory $\dabs{U_{\Psi}, G_{\Psi} \times \Gm, \bar W}_{\op{rel} U_{\Sigma, D_i}' }$.

The claim then follows from the fact that 
\begin{equation}\begin{aligned}\label{cat compo}
\Omega^{-1} \circ \varphi_*& \circ i^* \circ i_*\circ\varphi^* \circ\Omega \\&= 
\Omega^{-1}  \circ\varphi_* \circ \op{Id}_{\dabs{ U_{\Sigma_{-D_1,\dots,-D_r}} \times \Gm^{\Psi(1) \setminus \Sigma_{-D_1, \dots, -D_r}(1)} , S_{\Sigma_{\Psi}} \times \Gm, W}|_{\op{rel } U_\Psi}} \circ \varphi^* \circ \Omega \\ &= \op{Id}_{\op{Perf} [\mathcal Z/G]}.
\end{aligned}\end{equation}
and Definition~\ref{def: CCR}.
\end{proof}
\begin{remark}
Equation ~\eqref{cat compo} implies that ~\eqref{adjoint inclusions 2} is a crepant categorical resolution in the sense of Kuznetsov's definition \cite[Definition 3.2]{Kuz08} as well.
\end{remark}

We then obtain the following corollary when $[\mathcal{Z}/G]$ is smooth.

\begin{corollary}\label{cor CCR 1}
Consider the situation of Theorem~\ref{partial compact CCR} above. Assume further that $[\mathcal{Z}/G]$ is a smooth Deligne-Mumford stack. Then $i_* \circ \varphi^* \circ \Omega$ is a fully faithful functor. 
\end{corollary}
\begin{proof}
This follows directly from Theorem~\ref{partial compact CCR} and Lemma 4.24.4 of~\cite{stacks}.
\end{proof}

Furthermore, we have the following corollary we use later.

\begin{corollary}\label{cor: CCR equiv}
Suppose we are in the situation of Corollary~\ref{cor CCR 1}. Assume further that $\dabs{U_{\Psi}, G_{\Psi} \times \Gm, \bar W}$ is Calabi-Yau and connected. Then $i_* \circ \varphi^* \circ \Omega$ is an equivalence.
\end{corollary}
\begin{proof}
This follows directly from the fact that a Calabi-Yau category is indecomposable (see, e.g., Proposition 5.1 of \cite{KuzCY}).
\end{proof}

\section{Variation of GIT and the Exoflop}\label{sec: flop}

In the last section, we established a candidate for a crepant categorical resolution for the derived category of a toric complete intersection: the factorization category of a partial compactification of a toric vector bundle equipped with a superpotential extended onto the partial compactification. In the notation above, this is the category
$$
\dabs{U_{\Psi}, G_{\Psi} \times \Gm, \bar W}.
$$
This category is associated to the Cox construction associated to a fan $\Psi$ where $\Psi$ contains the toric vector bundle as a subfan.

An exoflop involves performing some flops after partially compactifying. Examples can be calculated using toric geometry when one uses toric geometric invariant theory and vary the stability parameter. To do so, we need $\mathcal{X}_{\Psi}$ to be a GIT quotient. In the first subsection, we give sufficient criteria for this.

In the toric case, the choices of GIT quotients are parameterized by the secondary fan, which parameterizes the choice of linearization for the GIT quotient. For each elementary wall crossing between maximal chambers in the secondary fan, there are two open sets $U, U' \subseteq \mathbb{A}^{\Psi(1)}$ so that $[U/G_{\Psi}]$ and $[U'/G_{\Psi}]$ are semiprojective. Moreover, there is an established relation between the categories $\dabs{U, G_{\Psi} \times \Gm, \bar W}$ and $\dabs{U', G_{\Psi} \times \Gm, \bar W}$, proven by Ballard-Favero-Katzarkov \cite[Theorem 3.5.2]{BFK} and Halpern-Leistner \cite[Proposition 4.2]{HL}. In certain scenarios, these results are made concrete for the entire secondary fan in \cite[Sections 4 and 5]{FK18}.  The results in loc. cit. make it possible to black box the geometric invariant theory, and we do so here. Since we black box this machinery, the flops may not be as transparent to the reader as they can be.

\subsection{Semiprojectivity of partial compactifications}\label{subsec:convex geometry} In this subsection, we discuss the existence of semiprojective partial compactifications by using regular triangulations and convex geometry. This subsection may at first seem like a digression, but it proves we have a partial compactification that is a GIT quotient. However, since the results of the subsection are true outside of the above context, we rename $N \times \Z^r$ to $N$ temporarily until Corollary~\ref{reg triang}. Said corollary is a strengthening of Lemma~\ref{ptl compactification} and the main result for our purposes in this subsection. 

There are two standard notions of regular triangulation in the literature. We start by reviewing them. Take a finite subset of distinct elements $\nu = \{v_1, \dots, v_r\} \in N_{\mathbb{Q}}$ lying on an integral affine hyperplane $H\subseteq N_{\R}$ with $0 \notin H$. This gives the lattice polytope $Q_\nu = \op{Conv}(v_1, \dots, v_r)\subseteq H$. We will assume that $Q_\nu$ has full dimension in $H$ and that the cone $C_\nu = \op{Cone}(\nu)$ has full dimension in $N_{\R}$ and is strongly convex. The below is a minor variant on the definition of triangulation from that in \cite[\S 15.2]{CLS}, which we follow.

\begin{definition}
A \newterm{triangulation} $\mathcal{T}$ of $\nu$ is a collection of simplices satisfying: 
\begin{itemize}
\item each simplex in $\mathcal{T}$ has codimension 1 in $N_{\R}$ with vertices in $\nu$;
\item the intersection of any two simplices in $\mathcal{T}$ is a face of each;
\item the union of the simplices in $\mathcal{T}$ is $Q_\nu$.
\end{itemize}

\end{definition}
One can define a special class of triangulations, \emph{regular} triangulations, as follows. Given nonnegative weights $\omega = (w_1, \dots, w_r) \in \mathbb{Q}^{r}_{\ge 0}$ (or, equivalently a weight function $w: \nu \to \Q_{\ge0}$) to obtain a cone 
$$
C_{\nu, \omega} := \op{Cone}( (v_1, w_1) \dots, (v_r, w_r)) \subseteq N_{\R} \times \R. 
$$
The lower hull of $C_{\nu, \omega}$ consists of all facets of the cone $C_{\nu, \omega}$ whose inner normal has a positive last coordinate. Projecting the facets in the lower hull and their faces gives a fan $\Sigma_{\omega}$ in $N_{\R}$ such that $|\Sigma_\omega| = C_{\nu}$ and $\Sigma_{\omega} \subseteq \{\op{Cone}(v_i) \ | \ 1 \le i \le r\}$. The fan $\Sigma_\omega$ naturally provides a polyhedral subdivision of the convex hull $Q_\nu$. The following is a variant of Definition 15.2.8 of \cite{CLS} (requiring rational weights $\omega \subseteq \mathbb{Q}^{r}_{\ge 0}$ rather than $\R^r_{\ge 0}$) for when this polyhedral subdivision is a triangulation.

\begin{definition}
A triangulation $\mathcal{T}$ of $\nu$ is \newterm{regular} if there are weights $\omega$ so that $\Sigma_{\omega}$ is simplicial and $\mathcal{T} = \Sigma_{\omega} \cap Q_{\nu}$. 
\end{definition}

There is an alternate definition of regular triangulation  by Hausel and Sturmfels for cones \cite{HS02}. We modify their definitions by adding the word `conical' to avoid confusion with the above.

\begin{definition}
 A \newterm{conical triangulation} of $\nu$ is a simplicial fan $\Sigma$ whose rays have generators in $\nu \subseteq N$. A \newterm{$T$-Cartier divisor} on $\Sigma$ is a continuous function $$ \Phi: C_{\nu} \to \R$$ which is linear on each cone of $\Sigma$ and takes integer values on $N\cap C_{\nu}$. The conical triangulation $\Sigma$ is called \newterm{regular} if there exists a $T$-Cartier divisor $\Phi$ which is ample, i.e., the function $\Phi$ is convex and restricts to a different linear function on each maximal cone of $\Sigma$. 
\end{definition}

\begin{proposition}\label{HS translation}
Suppose one has a regular triangulation of the point configuration $\nu = (v_1, \dots, v_r)$ with weights $\omega =(w_1, \dots, w_r) \in \mathbb{Q}_{\ge0}^r$. Then $\Sigma_\omega$ is a regular conical triangulation of $C_\nu$.
\end{proposition}

\begin{proof}
 For any $q \in Q_{\nu}$, there is a unique $w_q \in \mathbb{R}_{>0}$ such that $(q,w_q)$ is in the lower hull of $C_{\nu, \omega}$. Note that if $q \in Q_{\nu} \cap N$ then $w_q \in \mathbb{Q}_{>0}$ as it is a rational linear combination of $w_1, \dots, w_r$. Write $C_{\nu} = \{ a q \ | \ q \in Q_{\nu}, a \in \mathbb{R}_{\ge0}\}$. Then define a map $\Psi: C_{\nu} \to \mathbb{R}$
$$
\Phi: C_{\nu} \to \mathbb{R}_{\ge 0}, \qquad \Psi(aq) = aw_q.
$$
Using the fact that the weights $w_q$ are built from the lower hull of $C_{\nu,\omega}$, one can check this is continuous, convex, linear on each cone of $\Sigma_w$, and restricts to a different linear function on each maximal cone of $\Sigma_w$. One can then use Gordan's lemma on each maximal cone of $\Sigma_w$ to find a constant $D$ so that the function 
$$
\Phi_D: C_{\nu} \to \mathbb{R}_{\ge 0}, \qquad \Psi(aq) =D aw_q.
$$
satisfies the properties above and also takes integer values on $N\cap C_\nu$. 
\end{proof}

\begin{lemma}\label{ExtendRegtriang}
    Consider two finite sets of lattice points $L_0,L_1\subseteq N_{\R}$ and their union $L:=L_0\cup L_1$ such that $\dim\conv(L_0)=\dim\conv(L)$. Let $\mathcal{T}_0$ be a regular triangulation of $L_0$. Then there is a regular triangulation $\mathcal{T}$ of $L$ that contains $\mathcal{T}_0$ in the sense that every simplex $T\in\mathcal{T}_0$ is also contained in $\mathcal{T}$, i.e. $\mathcal{T}_0\subseteq\mathcal{T}$.
\end{lemma}

\begin{proof}
    We first prove the existence of a regular polyhedral subdivision $\mathcal{S}$ of $L$ containing $\mathcal{T}_0$ in the sense of the lemma, and then we prove that this regular subdivision can be refined into a regular triangulation containing $\mathcal{T}_0$, thus proving the Lemma. 
    
 If $L_1 \subset L_0$, the lemma is trivial. We proceed by induction. Suppose $L_1\setminus L_0 =\{v\}$. Denote the standard basis of $M_{\R}\times \R$ by $e_1, e_2,\dots, e_{d+1}$. By definition, the regular triangulation $\mathcal{T}_0$ of $L_0$ is obtained by projecting the lower facets of a polyhedron $\mathcal{Q}_0=\conv(\{x+w_0(x) e_{d+1}\vert x\in L_0\})\subseteq M_{\R}\times \R$ for some weight function $w_0:L_0\rightarrow \Q_{\ge0}$. To obtain the regular subdivision $\mathcal{S}$ of $L=L_0\cup \{v\}$ which contains $\mathcal{T}_0$, we extend $w_0$ to a weight function $w_1:L\rightarrow \R^+$ and project the lower facets of the polyhedron $\mathcal{Q}_1=\conv(\{x+w_1(x)e_{d+1}\vert x\in L\})\subseteq M_{\R}\times \R$, noting that $\mathcal{Q}_0\subseteq\mathcal{Q}_1$. We distinguish two cases: $v\in\conv(L_0)$ and $v\not\in\conv(L_0)$.
    
  \underline{Case 1: $v\in\conv(L_0)$.} Define the weight function $w_1:L\rightarrow \Q_{\ge0}$ by setting $w_1(v)=1+\max_{x\in L_0}(w_0(x))$ and $w_1(x)=w_0(x)$ otherwise. Since $v \in \conv(L_0)$, there exists weights $\lambda_x \in [0,1]$ for all $x \in L_0$ such that $\sum_{x \in L_0} \lambda_x = 1$ and $\sum_{x \in L_0} \lambda_x x = v$. Thus, for any such collection of weights,
  $$
  w_1(v) > \max_{x\in L_0}(w_0(x)) \ge \sum_{x\in L_0} \lambda_x w_0(x)
  $$
hence $(v,w_1(v))$ is not in any lower facet of the polyhedron $\mathcal{Q}_1$. Thus $\mathcal{S} = \mathcal{T}_0$ is a regular triangulation.

\underline{Case 2: $v\not\in\conv(L_0)$.} First we construct a weight function $w_1$. Fix a lower facet $F$ of $\mathcal{Q}_0$ with inner pointing normal $u_F$. Write $F=\mathcal{Q}_0\cap A_F$ where $A_F$ is the boundary of the supporting halfspace $H_F=\{x\in N_{\R}\times \R\vert \langle u_F, x\rangle \geq c_F\}$ at the facet $F$, so $A_F = \{x\in N_{\R}\times \R\vert \langle  u_F, x \rangle=c_F\}$. Write $\mu_F=\langle u_F, e_{d+1}\rangle >0$ and $\overline{u}_F=u_F-\mu_F e_{d+1}^*$, where $e_{d+1}^*$ is the dual basis vector to $e_{d+1}$ and $\overline{u}_F$ (respectively $\mu_Fe_{d+1}^*$) is the projection of $u_F$ onto the first $d$ coordinates (last coordinate). Define $w_1: L_1 \cup L_0 \to \Q_{\ge0}$ to be
\[
w_1(x) = \begin{cases}w_0(x) & \text{ if $x \in L_0$}; \\  1+\max_{x\in L_0}\frac{c_F-\langle \overline{u}_F, v\rangle}{\mu_F} & \text{ if $x = v$}.\end{cases}
\]

Consider the polyhedron $\mathcal{Q}_1=\conv(\{x+w_1(x)e_{d+1}\vert x\in L\})$. Note \[
\bigcup_{T \text{ a lower facet of }\mathcal{Q}_0} T\subseteq \mathcal{Q}_0\subseteq \mathcal{Q}_1.
\]
 Fix a lower facet $F=\mathcal{Q}_0\cap A_F$ of $\mathcal{Q}_0$. We claim that $F$ is a lower facet of $\mathcal{Q}_1$, i.e. $\mathcal{Q}_1\cap A_F=F$ and $\mathcal{Q}_1\subseteq H_F$.

 Suppose that $\mathcal{Q}_1\not\subseteq H_F$, i.e. there is a point $q\in \mathcal{Q}_1$ such that $q\not\in H_F$. Then $c_F > \langle u_F, q\rangle$ and there are non-negative real numbers $(\lambda_x)_{x\in L_0},\lambda_v$ where $1=\lambda_v+\sum_{x\in L_0}\lambda_x$ such that\[
 q=\lambda_v(v+w_1(v) e_{d+1})+\sum_{x\in L_0}\lambda_x(x+w_1(x)e_{d+1}).
 \]
 We obtain
  \begin{equation}\begin{aligned}\label{not on the hyperplane}
    c_F&>\left\langle u_F, \lambda_v(v+w_1(v)e_{d+1})+\sum_{x\in L_0}\lambda_x(x+w_1(x)e_{d+1})\right\rangle\\
     & =\lambda_v(\langle \overline{u}_F,v\rangle +w_1(v)\mu_F)+\left\langle u_F, \sum_{x\in L_0}\lambda_x(x+w_1(x)e_{d+1})\right\rangle\\
     & > \lambda_v\left( \langle \overline{u}_F,v\rangle+\frac{c_F-\langle \overline{u}_F,v\rangle}{\mu_F}\mu_F\right)+\sum_{x\in L_0}\lambda_xc_F\\
     & =c_F.
\end{aligned} \end{equation}
This is a contradiction, and so $\mathcal{Q}_1\subseteq H_F$.

Suppose, for the point of contradiction, there is a point $q\in(\mathcal{Q}_1\setminus\mathcal{Q}_0)\cap A_F$. Then there is a collection of non-negative real numbers $\{\lambda_x\vert x\in L_0\}\cup \{\lambda_v\}$ such that $\lambda_v+\sum_{x\in L_0}\lambda_x=1$ with $\lambda_v>0$ and $\lambda_v(v+w_1(v)e_{d+1})+\sum_{x\in L_0}\lambda_x(x+w_1(x)e_{d+1})=q$ and
 
 \begin{align*}
     c_F&=\langle u_F, q\rangle \\
     & =\langle u_F, \lambda_v(v+w_1(v)e_{d+1})\rangle +\langle u_F, \sum_{x\in L_0}\lambda_x(x+w_1(x)e_{d+1})\rangle .
 \end{align*}

 For each $x\in L_0$, we have $(x+w_1(x)e_{d+1})\in\mathcal{Q}_0$ so $\langle u_F, x+w_1(x)e_{d+1} \rangle \geq c_F$. Thus we obtain 
 
 \begin{align*}
     c_F& \geq \langle u_F, \lambda_v(v+w_1(v)e_{d+1})\rangle +c_F\sum_{x\in L_0}\lambda_x\\
     &= \lambda_v\langle u_F, v+w_1(v)e_{d+1}\rangle+(1-\lambda_v)c_F;\\
    \lambda_vc_F &  \geq \lambda_v\langle \overline{u}_F + \mu_F e_{d+1}^*, v+w_1(v)e_{d+1}\rangle;\\
 c_F-\langle \overline{u}_F, v\rangle &\geq \langle \mu_F e_{d+1}^*, w_1(v)e_{d+1}\rangle > \frac{c_F-\langle \overline{u}_F, v\rangle}{\mu_F}\mu_F=c_F-\langle \overline{u}_F, v\rangle;
 \end{align*}
 a contradiction. Hence, $(\mathcal{Q}_1\setminus\mathcal{Q}_0)\cap A_F=\emptyset$. Consequently, as $\mathcal{Q}_0\subseteq\mathcal{Q}_1$, we have $\mathcal{Q}_1\cap A_F=\mathcal{Q}_0\cap A_F$ as required. 

 In summary, we have shown that all lower facets of $\mathcal{Q}_0$ are also lower facets of $\mathcal{Q}_1$. Thus, \[
\mathcal{T}_0\subseteq \mathcal{S}.
 \]
 
Suppose  $\mathcal{S}$ is not a triangulation. Then there exist a lower facet $F = \mathcal{Q}_1 \cap A_F$ of $\mathcal{Q}_1$ that is not a simplex.  Recall that $\mathcal{T}_0$ is a triangulation, so $F \notin \mathcal{Q}_0$ and $v \in F$. Write 
$$
F = \op{Conv}(v+ w_1(v)e_{d+1}, (x_j + w_1(x_j) e_{d+1})_{j \in J})
$$
for some $J \subseteq L_0$. Assume $J$ is maximal in the sense that if $x + w_1(x) e_{d+1} \in F$ for some $x \in L_0$ then $x \in J$. Note that $\dim \op{Conv}(L_0) = \dim F \le |J| -1$, as $F$ is not a simplex. This would require there to be a facet $F' \subseteq\mathcal{Q}_0$ given by $F' = \op{Conv} (x_j + w_1(x_j) e_{d+1})_{j \in J}$ such that $ F' = \mathcal{Q}_0 \cap A_F$ and $\dim F = \dim F'$. Since $\mathcal{T}_0$ is a triangulation, $|J|-1 =\dim F' = \dim F= \dim \op{Conv}(L_0)$. This implies that the minimal affine linear subspace containing $ (x_j + w_1(x_j) e_{d+1})_{j \in J}$ must contain $v+ w_1(v)e_{d+1}$, which contradicts the fact that $\mathcal{Q}_0 \subset \mathcal{Q}_1$ (e.g., Equation~\eqref{not on the hyperplane}). 

The induction step now follows easily by iterating the above for each element of $L_1$. 
\end{proof}

The following is a strengthened version of Lemma~\ref{ptl compactification}. Recall the (strictly convex) rational polyhedral cone $\sigma'$ so that $ |\Sigma_{-D_1,\dots,-D_r}| \subseteq \sigma' \subseteq \sigma_{W}^\vee$. 

\begin{corollary}\label{reg triang}
There exists a simplicial fan $\Psi$ with support $\sigma'$ so that $\Sigma_{-D_1, \dots, -D_r}$ is a subfan of $\Psi$ and $X_{\Psi}$ is semiprojective.
\end{corollary}

\begin{proof}
Since $\sigma$ is a polyhedral strictly convex cone in $N_{\R} \times \R^r$, there exists an element $\bar m \in M_{\Q} \times \Q^r$ so that $\sigma \setminus \{0\}$ is contained in the halfspace $H:= \{ \bar n \in N_{\R} \times \R^r \ | \ \langle \bar m, \bar n \rangle>0\}$. Consider the set $L_0 = \{ v \in H_{\bar m}(1) \ | \ v \in \rho, \text{ for some } \rho \in \Sigma_{-D_1, \dots, -D_r}(1)\}$ and take $L_1$ to be the generators of the cone $\sigma$ that are contained in the halfspace $H$. Note $L=L_0 \cup L_1 \in N_{\Q} \times \Q^r$. Then by Lemma~\ref{ExtendRegtriang} we have a regular triangulation of $L$, for some weights $\omega$, yielding a simplicial fan $\Psi = \Sigma_{\omega}$. Thus by Proposition~\ref{HS translation}, the fan $\Psi$ is a regular conical triangulation. By Corollary 2.7 of \cite{HS02}, we obtain that $X_{\Psi}$ is semiprojective.
\end{proof}

\subsection{Variation of GIT Quotients and Gorenstein Cones}
 Corollary~\ref{reg triang} shows there exists a simplicial fan $\Psi$ with support $\sigma'$ so that $X_{\Psi}$ is semiprojective and $\Psi$ contains $\Sigma_{-D_1, \dots, -D_r}$ as a subfan. Recall that we have the global $G_{\Psi(1)}$-invariant function $\bar W$, defined as in ~\eqref{extended potential}, yielding the gauged LG model
$$
(U_\Psi, G_{\Psi(1)} \times \Gm, \bar W)
$$
Since $X_{\Psi}$ is semiprojective, $\Psi$ corresponds to a chamber of a secondary fan corresponding to the point collection given by intersecting a hyperplane with the rays in $\Psi(1)$ (see, e.g., Exercise 15.1.8 of \cite{CLS}). Thus the quotient stack $[U_{\Psi} / G_{\Psi}]$ is a GIT quotient and we can vary the choice of linearization to find other GIT quotients, and provide relationships between their corresponding factorization categories using the technology developed in \cite{BFK, HL}. Such a change will give a birational transformation, yielding the `flop' portion of the exoflop. One can use this technology directly and the above in order to generate results on relations in various cases.

The story simplifies when one restricts to Calabi-Yau complete intersections. In this case, the relationships between the factorization categories of different GIT quotients were streamlined in \cite{FK18} using Gorenstein cones. The following definitions for various variants of Gorenstein cones will be necessary for some of the following statements.

\begin{definition}
Consider a cone $\sigma$ in $N_{\R}$. We say $\sigma$ is 
\begin{enumerate}
    \item \newterm{Gorenstein with respect to $\mathfrak m_\sigma$} if there exists an element $\mathfrak m_\sigma \in M$ so that the cone $\sigma$ is generated over $\Q$ by finitely many lattice points in $\{n \in N \ | \ \langle \mathfrak m_\sigma, n \rangle =1 \}$.
    \item \newterm{$\Q$-Gorenstein with respect to $\mathfrak m_\sigma$} if there exists an element  $\mathfrak m_{\sigma} \in M_{\Q}$ so that the cone $\sigma$ is generated over $\Q$ by finitely many lattice points in $\{n \in N \ | \ \langle \mathfrak m_\sigma, n \rangle =1 \}$.
\end{enumerate}
 If $\sigma$ has primitive lattice generators $v_1, \dots, v_k \in N$, the \newterm{support} $\Delta_\sigma$ of $\sigma$ is the polytope $\op{Conv}(\{v_1, \dots, v_k\})$ in the hyperplane $$H_{\mathfrak m_\sigma}(1) := \{ x \in N_{\R} \ | \ \langle \mathfrak m_{\sigma}, x\rangle=1\}.$$
 We say a Gorenstein cone $\sigma$ is \newterm{reflexive Gorenstein of index $r$} if its dual cone $\sigma^\vee$ is Gorenstein with respect to an element $\mathfrak n_{\sigma^\vee} \in N$ and $\langle \mathfrak m_{\sigma}, \mathfrak n_{\sigma^\vee}\rangle = r$. 
\end{definition}
\begin{remark}
Here it is important to note that we do not require that $\sigma$ is generated over $\Z$ but over $\Q$. This is sometimes called almost Gorenstein and is a weaker assumption.
\end{remark}
\begin{definition}
    Given $r$ lattice polytopes $\Delta_1, \dots, \Delta_r \subseteq M_{\R}$, we define a \newterm{Cayley polytope of length $r$ associated to $\Delta_1, \dots, \Delta_r$} to be the convex hull $\op{Conv}(\Delta_1+e_1, \dots, \Delta_r + e_r) \subseteq M_{\R} \times \R^r$, where $e_i$ is the $i$th standard basis vector of $\R^r$. We say a cone is a \newterm{Cayley cone (of length $r$)} if it is the cone over a Cayley polytope of length $r$. We say a cone is \newterm{completely split} if it is a reflexive Gorenstein cone of index $r$ that is also a Cayley cone of length $r$.
\end{definition}

\noindent Note that if $\sigma$ is $\Q$-Gorenstein and dimension $\dim N_{\R} = d$, then $m_{\sigma}$ is unique. 

\begin{example}
Let $\Sigma$ be a complete fan in $N_{\R}$. Take the fan $\Sigma_{-D_1, \dots, -D_r}$ where $D_i = \sum_{\rho \in \Sigma(1)} a_{\rho i} D_\rho$ and, for each $i$, $a_{\rho i} = \delta_{ij}$ for some $j$ (that is, $\sum_i D_i = -K_{X_{\Sigma}}$ and the $D_i$ partition the anticanonical divisor). Then the support $|\Sigma_{-D_1, \dots, -D_r}|$ is an almost Gorenstein cone in $N_{\R} \times \R^r$ with respect to $m= e_1 +\dots + e_r$. 
\end{example}

Let $\sigma \subseteq N_{\R}$ be a $\Q$-Gorenstein cone and $\nu \subseteq \sigma \cap N$ be a finite, geometric collection of lattice points which contains the (primitive) ray generators of $\sigma$. Partition the set $\nu$ into two subsets
\begin{equation}\begin{aligned}
\nu_{=1} = \{v \in \nu \ | \ \langle m_{\sigma}, v\rangle = 1\} \text{ and }\\
\nu_{\ne 1} = \{ v \in \nu \ | \ \langle m_{\sigma}, v \rangle \ne 1\}. 
\end{aligned}\end{equation}
Note that since $\sigma$ is $\Q$-Gorenstein, the ray generators of $\sigma$ are contained in $\nu_{=1}$. 

We consider the fan $\Psi$ above. Suppose the cone $|\Psi|$ is $\Q$-Gorenstein. Then we have the following result.
\begin{theorem}[Theorem 5.8 of \cite{FK18}]\label{Thm:5.8inFK18}
Let $\Psi$ be any simplicial fan such that $\Psi(1) = \{ \op{Cone}(v) \ | \ v \in \nu\}$ and $X_{\Psi}$ is semiprojective. Similarly, let $\tilde \Sigma$ be any simplicial fan such that $\tilde\Sigma(1) \subseteq \nu_{=1}$, $X_{\tilde \Sigma}$ is semiprojective and $\op{Cone}(\tilde{\Sigma}(1)) = |\Psi|$. We have the following:
\begin{enumerate}
\item[(a)] If $\langle m_{\sigma}, a\rangle > 1$ for all $a \in \nu_{\ne 1}$, then there is a fully-faithful functor
$$
\dabs{U_{\tilde \Sigma} \times \Gm^{\nu \setminus \tilde\Sigma(1)}, G_{\Psi} \times \Gm, \bar W} \to \dabs{U_{\Psi}, G_{\Psi} \times \Gm, \bar W}.
$$
\item[(b)] If $\langle m_{\sigma}, a\rangle < 1$ for all $a \in \nu_{\ne 1}$, then there is a fully-faithful functor
$$
\dabs{U_{\Psi}, G_{\Psi} \times \Gm, \bar W} \to \dabs{U_{\tilde \Sigma} \times \Gm^{\nu \setminus \tilde\Sigma(1)}, G_{\Psi} \times \Gm, \bar W}.
$$
\item[(c)] If $\nu_{\ne 1} = \varnothing$, then there is an equivalence
$$
\dabs{U_{\tilde \Sigma} \times \Gm^{\nu \setminus \tilde\Sigma(1)}, G_{\Psi} \times \Gm, \bar W} \cong \dabs{U_{\Psi}, G_{\Psi} \times \Gm, \bar W}.
$$
\end{enumerate}
\end{theorem}

Using Theorem~\ref{Thm:5.8inFK18}, we are left to study the category 
$$
\dabs{U_{\tilde \Sigma} \times \Gm^{\nu \setminus \tilde\Sigma(1)}, G_{\Psi} \times \Gm, \bar W}.
$$
With some additional assumptions, we can prove that it is geometric. The following is half of \cite[Corollary 5.15]{FK18}.
\begin{proposition}\label{prop: regeometrize}
Take a fan $\tilde\Sigma$ as in Theorem~\ref{Thm:5.8inFK18}. Suppose there exist rays $\rho'_1, \dots, \rho'_r \in \tilde\Sigma(1)$ with primitive generators $e_{\rho'_1}, \dots, e_{\rho'_r} \in N \times \Z^r$ so that

\begin{enumerate}
\item  the induced projection
$$
\pi: N\times \Z^r \to (N\times \Z^r) / (\oplus_{i=1}^t \Z \cdot e_{\rho'_i})
$$
induces the toric morphism $\pi:  X_{\Sigma'_{-D_1', \dots, -D_r'}} \to X_{\Sigma'}$ and this toric morphism is a rank $r$ vector bundle whose sheaf of sections is $\oplus_{i=1}^r \mathcal{O}_{X_{\Sigma'}}(-D_i')$, and 
\item $e_{\rho'_1}+ \dots + e_{\rho'_r}= e_1 + \dots + e_r \in \N \times\Z^r$. 
\end{enumerate}
Write the function $\bar W$ as
$$
\bar W = \sum_{\bar m \in \Xi_{i,W}}  c_m \prod_{\rho \in \tilde\Sigma(1)} x_\rho^{\langle \bar m, u_{\rho}\rangle}.
$$
Then, there exists a partition $\Xi_{i,W} = H_1' \cup \dots \cup H_r'$ so that we can write
$$
\bar W = u_1'g_1 + \dots + u_r' g_r, \text{ where } g_i = \sum_{\bar m \in H_i'} \prod_{\rho \in \tilde\Sigma(1)\setminus\{\rho'_1, \dots, \rho_r'\}} x_\rho^{\langle \bar m, u_\rho\rangle}
$$
where $g_i \in \Gamma(X_{\Sigma'}, \mathcal{O}_{X_{\Sigma'}}(D_i'))$. Then we have the quotient stack 
\begin{equation}\label{eq: 2ndCI}
[\mathcal Z'/G']:= [Z(g_1, \dots, g_r) / S_{\tilde\Sigma(1)}] \subseteq [U_{\tilde \Sigma}/S_{\Sigma}]
\end{equation} where
$$
\dabs{U_{\tilde \Sigma} \times \Gm^{\nu \setminus \tilde\Sigma(1)}, G_{\Psi} \times \Gm, \bar W} \cong \dbcoh{[\mathcal Z'/G']}.
$$
\end{proposition}

\noindent Note that if $[\mathcal Z'/G']$ is smooth, then $\dbcoh{[\mathcal Z'/G']}$ is homologically smooth and proper. Also, if $[\mathcal Z'/G']$ is a Calabi-Yau orbifold, $\dbcoh{[\mathcal Z'/G']}$ is Calabi-Yau. These observations are useful for satisfying the hypotheses of Theorem~\ref{partial compact CCR} and Corollary~\ref{cor: CCR equiv}, respectively.

\subsection{Categorical ramifications of the exoflop}\label{subsec:summary}
In Section~\ref{sec: exo}, we established the following diagram for when one partially compactifies the gauged Landau-Ginzburg model corresponding to a complete intersection in a toric variety:

\[ \begin{tikzcd}
\dabs{ U_{\Sigma_{-D_1,\dots,-D_r}}, S_{\Sigma_{-D_1,\dots,-D_r}(1)} \times \Gm, W} \arrow[r, shift left, "i_* \circ \varphi^*"]  & \dabs{U_{\Psi}, G_{\Psi} \times \Gm, \bar W}  \arrow[l, shift left, "\varphi_*\circ i^*"]  \\
\dbcoh{[\mathcal{Z}/G]}   \arrow[swap]{u}{\Omega} & 
\end{tikzcd}
\]
Here,  $\Omega$ is the equivalence in Corollary~\ref{cor: Hirano} and $i_* \circ \varphi^*$ 
\begin{itemize}
\item[(i)] with $\varphi_*\circ i^*$ forms a crepant categorical resolution if $\dabs{U_{\Psi}, G_{\Psi} \times \Gm, \bar W}$ is homologically smooth and proper (Theorem~\ref{partial compact CCR});
\item[(ii)] is a fully faithful functor if  $\dabs{U_{\Psi}, G_{\Psi} \times \Gm, \bar W}$ is homologically smooth and proper and $\dbcoh{[\mathcal{Z}/G]} $ is a smooth Deligne-Mumford stack (Corollary~\ref{cor CCR 1}); and
\item[(iii)] is an equivalence if  $\dabs{U_{\Psi}, G_{\Psi} \times \Gm, \bar W}$ is homologically smooth, proper, connected, and Calabi-Yau, and $\dbcoh{[\mathcal{Z}/G]} $ is a smooth Deligne-Mumford stack (Corollary~\ref{cor: CCR equiv}).
\end{itemize}

In this section, we established a chain of relations:
\[ \begin{tikzcd}
\dabs{U_{\Psi}, G_{\Psi} \times \Gm, \bar W} \arrow[r, leftrightarrow, dashed, "Thm.~\ref{Thm:5.8inFK18}"]& \dabs{U_{\tilde \Sigma} \times \Gm^{\nu \setminus \tilde\Sigma(1)}, G_{\Psi} \times \Gm, \bar W} \arrow[d, "Prop.~\ref{prop: regeometrize}"] \\
& \dbcoh{[\mathcal Z'/G']}
\end{tikzcd}
\]
Putting together, we have the following corollary that summarizes the exoflop's power:
\begin{corollary}\label{cor: exoflop}
\begin{itemize}
\item[(a)] If the conditions of Proposition~\ref{prop: regeometrize} hold, we are in either cases (a) or (c) of Theorem \ref{Thm:5.8inFK18}, and $[\mathcal Z'/G']$ is smooth, then we have that $\dabs{U_{\Psi}, G_{\Psi} \times \Gm, \bar W}$ is homologically smooth and proper, hence $i_* \circ \varphi^*\circ\Omega$ and $\Omega^{-1}\circ \varphi_*\circ i^*$ form a categorical resolution for $\dbcoh{[\mathcal{Z}/G]}$.
\item[(b)] If the conditions of Proposition~\ref{prop: regeometrize} hold, we are in either cases (b) or (c) of Theorem \ref{Thm:5.8inFK18}, and $[\mathcal Z/G]$ and $[\mathcal Z'/G']$ are smooth, then we have a fully faithful functor
$$
\dbcoh{[\mathcal Z/G]} \rightarrow \dbcoh{[\mathcal Z'/G']}.
$$
\item[(c)] If the conditions of Proposition~\ref{prop: regeometrize} hold, we are in case (c) of Theorem \ref{Thm:5.8inFK18}, $[\mathcal Z/G]$ is smooth, and $[\mathcal Z'/G']$ is a smooth, connected Calabi-Yau orbifold, then we have an equivalence
$$
\dbcoh{[\mathcal Z/G]} \stackrel{\sim}{\rightarrow} \dbcoh{[\mathcal Z'/G']}.
$$
\end{itemize}
\end{corollary}
There are many conditions at play in Corollary~\ref{cor: exoflop}. However, to illustrate its power, we describe in the next section combinatorial sufficient conditions using reflexive completely split Gorenstein cones.

\section{Exoflops for CICYs}\label{sec: Cicy}

In the previous sections, we aimed to provide general results. In this section, we specialize to the case of Calabi-Yau complete intersections (CICYs) in toric Fano varieties. We provide combinatorial context for when one can use Corollary~\ref{cor: exoflop}(a) and (c). We recall notation from above, but will additional assumptions in our set-up.

Let $X_{\Sigma}$ be a toric projective Fano variety and let $D_1, \dots, D_r$ be torus-invariant Weil divisors so that we can write
$$
D_{i} = \sum_{\rho \in \Sigma(1)} a_{\rho i } D_\rho
$$
where $a_{\rho i} = \delta_{ij}$ for some $j \in \{1, \dots, r\}$. Note $\sum_{i=1}^r D_i = -K_{X_{\Sigma}}$. Consider the toric fan $\Sigma_{-D_1, \dots, -D_r}$ as above. Note $|\Sigma_{-D_1, \dots, -D_r}|$ is almost Gorenstein with respect to 
\begin{equation}
\m := e_1^* + \dots + e_r^* \in M \times \Z^r.
\end{equation}
Write $\mathfrak{n} = e_1 + \dots + e_r$. 

Let $f_i \in \Gamma(X_{\Sigma}, \mathcal{O}_{X_\Sigma}(D_i))$ be a non-zero global section of $D_i$. We can write
\begin{equation}\label{eq:W}
W = u_1f_1 + \dots + u_r f_r = \sum_{\bar m \in H_{\mathfrak{n}}(1)\cap (M\times\Z^r)\cap |\Sigma_{-D_1, \dots, -D_r}|^\vee} c_{\bar m} \prod_{\rho \in \Sigma_{-D_1, \dots, -D_r}(1)} x_\rho^{\langle \bar m, u_\rho\rangle}
\end{equation}
for some $c_{\bar m} \in \C$. Define the set and cone
\begin{equation}\label{sigma w}
\Xi_W := \{ \bar m \in  H_{\mathfrak{n}}(1)\cap (M\times\Z^r)\cap |\Sigma_{-D_1, \dots, -D_r}|^\vee \ | \ c_{\bar m} \ne 0\}; \qquad \sigma_W := \op{Cone}(\Xi_W).
\end{equation}
 As above, take a cone $\sigma'$ so that $\sigma_W^\vee\supseteq \sigma' \supseteq |\Sigma_{-D_1, \dots, -D_r}|$. 

\begin{assumption}\label{conical assumption}
We assume the following: 
\begin{enumerate}
\item[(i)] The cone $\sigma'$ is almost Gorenstein with respect to $\m$.
\item[(ii)] There exists a fan  $\Sigma'_{-D_1', \dots, -D_r'}$ with support $|\Sigma'_{-D_1', \dots, -D_r'}| = \sigma'$ so that:
\begin{itemize}
\item for any primitive generator $u_{\rho'}$ of a ray $\rho' \in \Sigma'_{-D_1', \dots, -D_r'}(1)$ we have $\langle \m, u_{\rho'}\rangle =1$;
\item there exists rays $\rho_1', \dots, \rho'_r \in \Sigma'_{-D_1', \dots, -D_r'}(1)$ so that 
\begin{itemize}
\item $u_{\rho'_1} + \dots + u_{\rho'_r} = \mathfrak{n}$, i.e., $\sigma'$ is a Cayley cone associated to $r$ lattice polytopes and 
\item the projection $\pi: N\times \Z^r \to N\times \Z^r / (\oplus_{i=1}^r \Z \cdot u_{\rho'_i})$ induces a toric morphism  $\pi:  X_{\Sigma'_{-D_1', \dots, -D_r'}} \to X_{\Sigma'}$ for some fan $\Sigma'$ in $N_{\R}\times \R^r / (\oplus_{i=1}^r \R \cdot u_{\rho'_i})$ corresponding to a toric Fano variety $X_{\Sigma'}$ and this toric morphism is a rank $r$ vector bundle whose sheaf of sections is $\oplus_{i=1}^r \mathcal{O}_{X_{\Sigma'}}(-D_i')$.
\end{itemize}
\end{itemize}
\end{enumerate}
\end{assumption}
We denote by $v_1, \dots, v_r$ the variables corresponding to the rays $\rho_1', \dots, \rho'_r$. Since each monomial of $\bar W$ is of the form $x^{\bar m}$, $\langle \bar m, \mathfrak{n}\rangle=1$, and $\langle \bar m, u_{\rho'_i}\rangle \in \Z_{\ge 0}$, we can write the extended global function as
$$
\bar W = \sum_{\bar m \in \Xi_W} c_{\bar m} \prod_{\rho' \in \Sigma'_{-D_1', \dots, -D_r'}(1)} x_\rho^{\langle \bar m, u_{\rho'}\rangle} = v_1 g_1 + \dots + v_r g_r,
$$
where $g_i \in \Gamma(X_{\Sigma'}, \mathcal{O}_{X_{\Sigma'}}(D_i'))$. We then have $
[\mathcal Z'/G']:= [Z(g_1, \dots, g_r) / S_{\tilde\Sigma(1)}] $ as in ~\eqref{eq: 2ndCI}. 

\begin{corollary}\label{2 CICYs} Suppose Assumption~\ref{conical assumption} holds and take $[\mathcal Z/G], [\mathcal Z'/G']$ defined above. Then
\begin{itemize}
\item[(i)] If $[\mathcal Z'/G']$ is smooth, then we have a crepant categorical resolution 
\begin{align*}
F&:\dbcoh{[\mathcal Z'/G']} \to \dbcoh{[\mathcal Z/G]},   \\
G&:\op{Perf} [\mathcal Z/G] \to \dbcoh{[\mathcal Z'/G']}.
\end{align*}
\item[(ii)] If both $[\mathcal Z/G]$ and $[\mathcal Z'/G']$ are smooth and connected, then they are derived equivalent.
\end{itemize}
\end{corollary}
\begin{proof}
This is an application of Corollary~\ref{cor: exoflop}. If Assumption~\ref{conical assumption} holds, then the conditions of Proposition~\ref{prop: regeometrize} hold and we are in case (c) of Theorem~\ref{Thm:5.8inFK18}. Since the fan $\Sigma'_{-D_1', \dots, -D_r'}$ has only rays that pair to $1$ with $\m$, we have that $\dbcoh{[\mathcal Z'/G']} $ is Calabi-Yau by \cite[Corollary 5.15]{FK18}.
\end{proof}

The choice of $\sigma'$ is key for the possibility that Assumption~\ref{conical assumption} can hold. For this reason, we now provide combinatorial criteria from convex geometry.

\subsection{Gorenstein Cones and CICYs}

 The following is a general result about almost Gorenstein cones. 

\begin{proposition}
\label{Prop:Both csrG implies Ass1}
 Let  $\sigma\subseteq N_{\R}\times \R^r$ be an almost Gorenstein cone with respect to $\mathfrak{m}$ above. If both $\sigma$ and $\sigma^\vee$ are completely split reflexive Gorenstein of index $r$,  then $\sigma$ fulfills Assumption \ref{conical assumption}.
\end{proposition}
\begin{proof}
 Being reflexive Gorenstein with respect to $\mathfrak{m}$ implies that $\sigma$ is almost Gorenstein with respect to $\mathfrak{m}$.
    Firstly, we note that $\sigma$ and $\sigma^\vee$ both being completely split reflexive Gorenstein is equivalent to $\sigma$ being associated to a nef-partition by Corollary 3.7 of \cite{BN08}. That is, there exists some $e_1', \dots, e_r' \in N \times \Z^r$ that form part of a $\Z$-basis for $N\times \Z^r$ so that there exists a nef partition $\Delta_1+\dots+\Delta_r=\Delta$ with unique interior point $0$ in $N' := N\times \Z^r / (\oplus_{i=1}^r \Z e_i')$ and   $\sigma=\cone(\Delta_1+ e_1', \dots, \Delta_r + e_r')$. 

    Denote by $V_i$ the vertex set of $\Delta_i + e_i'$. Note $V_i\cap V_j=\emptyset$. Let $V=\bigcup_{i=1}^r V_i$. We note $\sigma=\cone(V)$, and by associating to each $p\in V$ the ray $\rho_p$ with primitive generator $p$, we have $\sigma(1)=\{\rho_p\ \vert \ p\in V\}$. We want to show that $\sigma=|\Sigma'_{-D_1',\dots,-D_r'}|$ for some vector bundle over a simplicial $\Sigma'$. We prove this by direct construction.

    Write $\pi: N_{\R} \times \R^r \to N_{\R}'$ for the projection and $\overline{\rho}_p := \pi (\rho_p)$. We have that the cone over $\pi(V) \setminus\{0\}$ has support $N_{\R}'$ as $e_1'+\dots+e_r'$ is in the relative interior of $\sigma$. Thus, there exists a complete fan in $N_{\R}'$ with rays $ \{\overline{\rho}_p \ | \ p \in V, p \ne e_i' \text{ for all $i$}\}$. One can then simplicially subdivide to obtain a fan $\Sigma'$.  Let 
    $$D'_i:=\sum_{\substack{ p\in V_i\\ p \ne e_i'\text{ for all $i$}}}D_{\overline{\rho}_p}.$$

    Note that, by Corollary 3.17 of \cite{BN08}, the images of $V_i$ and $V_j$ in $N'$ intersect only at the origin, thus each $D_{\rho_p}$ appears as a nontrivial summand in a unique divisor $D_i'$. Following the standard toric vector bundle construction \cite[\textsection7.3]{CLS}, we find that the vector bundle $\bigoplus \O_{X_{\Sigma'}}(-D'_i)$ has a fan $\Sigma'_{-D'_1,\dots,D'_r}$ with rays $\{\rho_p\ \vert \ p\in V\}$, i.e. $\Sigma'_{-D'_1,\dots,D'_r}(1)=\sigma(1)$, implying that $\sigma=|\Sigma'_{-D'_1,\dots,-D'_r}|$. As desired, we therefore have constructed directly a fan $\Sigma'_{-D'_1,\dots,-D_r'}$ that fulfills the conditions of Assumption \ref{conical assumption}.
\end{proof}

We look to apply Corollary \ref{2 CICYs} while using the above Proposition. To do so, it is sufficient to check if the dual cone $\sigma_W^\vee$ to $\sigma_W$ defined in~\eqref{sigma w} is completely split reflexive Gorenstein, and that the LG model corresponds to a smooth complete intersection. In the rest of the section, we find combinatorial criteria and genericity hypotheses where both are satisfied.

\begin{lemma}
\label{Lem:sigmaW dual csrG gives Assumption}
    Let $\sigma_W$ be as in~\eqref{sigma w}. If its dual $\sigma_W^\vee$ is completely split Gorenstein of index $r$, then $\sigma_W^\vee$ fulfills Assumption \ref{conical assumption}.
\end{lemma}
\begin{proof}
    The cone $\sigma_W^\vee$ is reflexive Gorenstein of index $r$, hence so is $\sigma_W$. The containment $|\Sigma_{-D_1,\dots,-D_r}|^\vee\supseteq \sigma_W$ implies that $\mathfrak{n}_{\sigma_W}$ is the Gorenstein element $\mathfrak{n}$ of $|\Sigma_{-D_1,\dots,-D_r}|^\vee$.

Furthermore, since $|\Sigma_{-D_1,\dots,-D_r}|$ is completely split and $|\Sigma_{-D_1,\dots,-D_r}|\subseteq \sigma_W^\vee$, there are elements $e_1^*, \dots, e_r^* 
 \in|\Sigma_{-D_1,\dots,-D_r}|\subseteq \sigma_W^\vee $ so that $e_1^* + \dots +e_r^*= \mathfrak n_{|\Sigma_{-D_1,\dots,-D_r}|}=\mathfrak{n}_{\sigma_W}$. Proposition 2.3 in \cite{BN08} then implies that $\sigma_W$ is a Cayley cone and thus a completely split reflexive Gorenstein cone of index $r$. Since $\sigma_W$ and $\sigma_W^\vee$ are both completely split reflexive Gorenstein cones of index $r$, Proposition \ref{Prop:Both csrG implies Ass1} implies that $\sigma_W^\vee$ fulfills the Assumption \ref{conical assumption}.
\end{proof}

The next result uses Lemma \ref{Lem:sigmaW dual csrG gives Assumption} and Bertini's theorem to allow us to apply Corollary \ref{2 CICYs}, hence providing a crepant categorical resolution as desired. First, let us set up some notation.

\begin{definition}
    We say $\Xi$ is \emph{saturated} if $\Xi = \op{Conv}(\Xi) \cap (M\times \Z^r)$.
\end{definition}

\begin{notation}
    Let $\Xi$ be saturated and $\Psi$ a fan so that $\Xi \subseteq |\Psi|^\vee$. We write $\mathcal{F}_{\Xi}$ for the family of polynomials $W: U_{\Psi} \to \mathbb{A}^1$ of the form
    $$
    W = \sum_{m \in \Xi} c_m \prod_{\rho\in \Psi(1)} x^{\langle m, u_\rho\rangle }.
    $$
\end{notation}

\begin{corollary}
\label{Cor:csrG gives cat resn}
Let $\Xi_W$ and $\sigma_W$ be as defined in~\eqref{sigma w} and suppose $[\mathcal Z / G]$ is positive dimensional. Suppose $\Xi_W$ is saturated, $\sigma_W^\vee$ is completely split reflexive Gorenstein of index $r$ and that $W\in \mathcal F_\Xi$ is sufficiently generic. Then there is a crepant categorical resolution of $[\mathcal Z/G]$
\begin{align*}
F&:\dbcoh{[\mathcal Z'/G']} \to \dbcoh{[\mathcal Z/G]},   \\
G&:\op{Perf} [\mathcal Z/G] \to \dbcoh{[\mathcal Z'/G']}
\end{align*}
by $[\mathcal Z'/G']$ as in~\eqref{eq: 2ndCI}. Moreover, if $[\mathcal{Z} / G]$ is smooth and $\dim [\mathcal{Z} / G]>0$, then there is a derived equivalence between $[\mathcal{Z}/G]$ and $[\mathcal{Z}'/G']$.
\end{corollary}
\begin{proof}
    By Lemma \ref{Lem:sigmaW dual csrG gives Assumption}, the cone $\sigma_W^\vee$ fulfills Assumption \ref{conical assumption}. To apply Corollary \ref{2 CICYs} and obtain the desired categorical resolution, it remains to show that the complete intersection $[\mathcal Z'/G']$ in $\mathcal X_{\Sigma'}$ is indeed smooth. In the proof of Proposition \ref{Prop:Both csrG implies Ass1}, we have shown that the lattice polytopes giving $\sigma_W^\vee$ its Cayley structure in fact give a nef-partition of $\Delta_{-K_{\Sigma'}}$.
    Hence, the divisors $D'_i$ corresponding to the nef-partition and giving the vector bundle $\bigoplus_{i=1}^r \O_{X_{\Sigma'}}(-D'_i)$ are nef. In particular, by Proposition 6.3.12 in \cite{CLS} these divisors are basepoint free. 
    Recall that any section $g\in \Gamma(\bigoplus \O_{X_{\Sigma'}}(-D_i'))$ can be expressed via a sum of monomials \[\sum_{m\in H_{\mathfrak{n}}(1)\cap (M\times\Z^r)\cap |\Psi|^\vee\cap M}c_mx^m\] for some coefficients $c_m$. By Bertini's Theorem, a generic section $(g_1,\dots,g_r)$ of the vector bundle will give a smooth complete intersection $[\mathcal Z'/G']:=Z(g_i)\subseteq \mathcal{X}_{\Sigma'}$. The set of such generic sections is open and dense in the linear system spanned by the divisors $D_i'$. As $\Xi_W$ is saturated, the open and dense set of generic sections must thus intersect the family corresponding to sections of the form $\sum_{m\in \Xi_W}c_mx^m$, i.e. there is an element $W$ in $\mathcal F_W$ such that the corresponding complete intersection $[\mathcal Z'/G']\subseteq \mathcal{X}_{\Sigma'}$ is smooth. Corollary \ref{2 CICYs} then gives the crepant categorical resolution  as desired. We remark that these categories are connected when the Calabi-Yau orbifolds are positive dimensional and thus the equivalence holds if the additional hypotheses are satisfied.
    \end{proof}

The Corollary \ref{Cor:csrG gives cat resn} gives us combinatorial conditions we can check to generate categorical resolutions as in Corollary \ref{2 CICYs}. Since this is useful, we give below another formulation of it that may be more user-friendly for applications.

\begin{definition}
    \label{Def:Int closed}
    A lattice polytope $\Delta\subseteq N_{\R}$ is called \newterm{integrally closed}, if any lattice point in the Gorenstein cone $\sigma$ over $\Delta$ is a sum of lattice points from the support $\sigma(1)\simeq \Delta$.
\end{definition}

\noindent One can rewrite Corollary~\ref{Cor:csrG gives cat resn} in terms of the support polytopes in the following way.

\begin{corollary}
    \label{Cor: CCR via integrally closed}
     Suppose $\Xi$ is saturated and that $\tilde{\Delta}=\conv(\Xi_W)$ and its dual $\tilde{\Delta}^\vee$ are both integrally closed Gorenstein polytopes (of index $r$). Then, for a generic polynomial $W$ in the family $\mathcal F_\Xi$, there is a crepant categorical resolution of $[\mathcal Z/G]$
\begin{align*}
F&:\dbcoh{[\mathcal Z'/G']} \to \dbcoh{[\mathcal Z/G]},   \\
G&:\op{Perf} [\mathcal Z/G] \to \dbcoh{[\mathcal Z'/G']}
\end{align*}
by $[\mathcal Z'/G']$ as in~\eqref{eq: 2ndCI}.
\end{corollary}
\begin{proof}
    By Corollary \ref{Cor:csrG gives cat resn} it is sufficient to show that $\sigma_W$ is completely split reflexive Gorenstein. Since $\tilde{\Delta}$ is a Gorenstein polytope, by Proposition 2.11 in \cite{BB97} the cone $\sigma_W$ is reflexive Gorenstein. Since both $\tilde{\Delta}$ and $\tilde{\Delta}^\vee$ are integrally closed, by Corollary 3.9 of \cite{BN08} we obtain that $\sigma_W$ is completely split and associated to a nef-partition. As $\Xi$ is saturated, there is a sufficiently generic polynomial in the family $\mathcal F_\Xi$ and the statement of the Corollary is not empty.
\end{proof}

\section{Examples and applications}
The following highlights use-cases and examples to build intuition on exoflops.

\subsection{Aspinwall's example}\label{subsec:Aspinwall}

We first explain an explicit example. We choose to repeat Aspinwall's primary example in his paper \cite[\textsection 2.3, 3.1-5]{Aspinwall}. In some sense this is a simple case in comparison to the general case of what can happen above as the partial compactification in the exoflop is as simple as possible, but it provides intuition on why an exoflop can give rise to a categorical resolution.

Consider a quartic 
$$
f = f_4(x_1, x_2, x_3) + x_0 f_3(x_1,x_2,x_3)+ x_0^2 f_2(x_1,x_2,x_3)
$$
where $f_k(x_1, x_2, x_3)$ are homogeneous equations of degree $k$. We assume that the $f_k$ are generic enough to avoid additional singularities. We have that $f \in \Gamma(\P^3, \mathcal{O}_{\P^3}(4))$, and $Z(f)\subseteq \P^3$ is a singular quartic surface. Indeed, its singular locus is the point $(1:0:0:0)$. 

The vector bundle $\op{tot} \O_{\P^3}(-4)$ can be written as a quotient stack $[(\A^4 \setminus\{0\}) \times \A^1 / \Gm]$, where $\Gm$ acts with weights $(1,1,1,1,-4)$. Write $u$ for the variable corresponding to the last coordinate. There is a $\Gm$-invariant global function 
$$
W=uf: (\A^4 \setminus\{0\}) \times \A^1 \to \A^1.
$$
One computes that the critical locus has two irreducible components, when $u=f=0$ and $Z(x_1,x_2, x_3)$. The former component, when viewed in the stack $[(\A^4 \setminus\{0\}) \times \A^1 / \Gm]$ is proper and isomorphic to $Z(f)$ in the zero section $u=0$. The latter is the $\A^1$ corresponding to the fiber over the point $(1:0:0:0)$. 

We construct an partial compactification of the quotient stack $[(\A^4 \setminus\{0\}) \times \A^1 / \Gm]$ and extend $W$. To do so, there is a stack isomorphism
$$
\varphi: [(\A^4 \setminus\{0\}) \times \A^1 / \Gm]  \to [(\A^4 \setminus\{0\}) \times \A^1 \times \Gm / (\Gm)^2]
$$
where the two $\Gm$ act by weights $(1,1,1,1,-4,0)$ and $(1,0,0,0,-2,1)$. We will use the variable $y$ for this new coordinate. Consider the $(\Gm)^2$-equivariant open immersion 
$$
i: (\A^4 \setminus\{0\}) \times \A^1 \times \Gm) \hookrightarrow \A^6 \setminus Z(yx_0, yx_1, yx_2, yx_3, ux_0).
$$
Here, we can extend $W$ by taking the $(\Gm)^2$-invariant function 
$$
\bar W= u(\bar f), \text{ where } \bar f = y^2f_4(x_1, x_2, x_3) + yx_0 f_3(x_1,x_2,x_3)+ x_0^2 f_2(x_1,x_2,x_3)
$$
Since the additional strata added by the open immersion is away from the zero section of the line bundle, the first component of the critical locus of $\bar W$ is the same as $W$. However, the second is compactified to be a weighted projective line.

\begin{remark}
The above construction is simple in toric geometry. Take the standard fan for $\P^3$. The vector bundle is the toric variety associated to the fan obtained by the star subdivision at the ray generated by the lattice point $(0,0,0,1)$ of the cone $$\op{Cone}((1,0,0,1),(0,1,0,1),(0,0,1,1), (-1,-1,-1,1)).$$ The partial compactification is found by adding the cone $$\op{Cone}((-1,0,0,1),(0,1,0,1),(0,0,1,1), (-1,-1,-1,1))$$ to the fan. The maximal special linear system allowed to take this partial compactification corresponds to 
\begin{align*}
    \Xi =  (M \times \Z) \cap \op{Conv}(&(-1,-1,-1,1), (-1,3,-1,1), (-1,-1,3,1), \\ &(1,-1,-1,1), (1,1,-1,1),(1,-1,1,1))
\end{align*}
Taking a generic enough potential using $\Xi$ is equivalent to choosing $f_4, f_3$, and $f_2$ above generic enough to avoid singularities.
\end{remark}

There is a toric flop (given by GIT) corresponding to the following birational map
$$
\psi: \A^6 \setminus Z(yx_0, yx_1, yx_2, yx_3, ux_0) \dashrightarrow \A^6 \setminus Z(x_0x_1, x_0x_2, x_0x_3, yx_1, yx_2, yx_3)
$$
One can check that $[\A^6 \setminus Z(x_0x_1, x_0x_2, x_0x_3, yx_1, yx_2, yx_3) / (\Gm)^2]$ is total space of the anticanonical bundle of the blow up $\op{Bl}_{(1:0:0:0)}\P^3$ of $\P^3$ at $(1:0:0:0)$, and that $\bar f \in \Gamma(\op{Bl}_{(1:0:0:0)}\P^3, -K_{\op{Bl}_{(1:0:0:0)}\P^3})$. Since $f$ was chosen sufficiently generically, $Z(\bar f)$ is smooth. One then obtains $Z(\bar f)\subseteq\op{Bl}_{(1:0:0:0)}\P^3 $ is a categorical resolution of $Z(f)\subseteq \P^3$ (as an example of Corollary~\ref{2 CICYs}). 

\subsection{Derived equivalences with varying bundle structures}\label{exa:diffbund}

As seen above, a standard (toric) resolution of singularities can appear from an exoflop, but there are some derived equivalences that are found where a birational equivalence is not obvious.  In this section we exhibit the convex geometry that leads to such a derived equivalence. This involves when the toric vector bundle structures differ (that is, there are different sets of minimal generators that sum to $\mathfrak{n}$, as seen in Assumption~\ref{conical assumption}). We will consider three different Calabi-Yau complete intersections (we have chosen a small dimensional case to attempt to not cloud the example with too much unnecessary toric geometry). 

For this example, we work in $N= \Z^5$ and take the rays $\rho_1, \dots, \rho_{12}$ with minimal generators
\begin{align*}
&u_{\rho_1}= (2,0,-1,0,1), \hspace{0.5em} &u_{\rho_2} = (0,2,-1,0,1), \hspace{0.5em} &u_{\rho_3}=(-1,-1,2,1,0), \hspace{0.5em} &u_{\rho_4} = (-1,-1,0,1,0), \\
&u_{\rho_5} = (1,-1,0,1,0), \hspace{0.5em} &u_{\rho_6} = (-1,1,0,1,0), \hspace{0.5em} &u_{\rho_7}=(0,0,1,0,1), \hspace{0.5em} &u_{\rho_8} = (0,0,-1,0,1), \\
&u_{\rho_9} = (0,-1,0,1,0), \hspace{0.5em} &u_{\rho_{10}} = (0,1,0,0,1), \hspace{0.5em} &u_{\rho_{11}} = (0,0,0,1,0), \hspace{0.5em} &u_{\rho_{12}}=(0,0,0,0,1).
\end{align*}
\noindent Consider the following three cones:
\begin{align*}
    & \sigma = \cone(\rho_1,\dots,\rho_8),\\
    & \sigma_1=\cone(\rho_1,\rho_2,\rho_3,\rho_4,\rho_{11},\rho_{12}),\\
    & \sigma_2=\cone(\rho_1,\rho_2,\rho_3,\rho_4, \rho_9, \rho_{10}).
\end{align*}
We refer the reader to Figure~\ref{fig} to get some geometric intuition in the convex geometry. Note that $\rho_9, \rho_{10}, \rho_{11}, \rho_{12} \in \sigma$, hence $\sigma_1,\sigma_2\subseteq \sigma$. Moreover, all 3 cones are completely split Gorenstein cones of index 2 with respect to $\mathfrak{m} = (0,0,0,1,1)$.

\begin{figure}
\centering
\begin{tikzpicture}
		[cube/.style={thick,black}, tetra/.style={thick,magenta}
			grid/.style={very thin,gray},
			axis/.style={->,blue,thick}, tetra/.style={thick,magenta}, common/.style={violet}]

			

	\draw[tetra] (2,0,-1) -- (0,2,-1) -- (0,0,-1) -- cycle;
	\draw[tetra] (2,0,-1) -- (0,0,1);
	\draw[tetra] (0,2,-1) -- (0,0,1);
	\draw[tetra] (0,0,-1) -- (0,0,1);
	
	\draw[tetra] (2,0,-1) node[right] {$u_1$};
	\draw[tetra] (2,0,-1) node {$\bullet$};
	
	\draw[tetra] (0,2,-1) node[right] {$u_2$};
	\draw[tetra] (0,2,-1) node {$\bullet$};
	
	\draw[tetra] (0,0,1) node[below] {$u_7$};
	\draw[tetra] (0,0,1) node {$\bullet$};
	
	\draw[tetra] (.25,0,-1) node[above] {$u_8$};
	\draw[tetra] (0,0,-1) node {$\bullet$};
	
	\draw[tetra] (0,1,0) node[left] {$u_{10}$};
	\draw[tetra] (0,1,0) node {$\bullet$};

	\draw[tetra] (0,0,0) node {$\bullet$};
	
	\draw[cube] (-1,-1,0) -- (-1,1,0) -- (1,-1,0) -- cycle;
	\draw[cube] (-1,-1,2) -- (-1,-1,0);
	\draw[cube] (-1,-1,2) -- (-1,1,0);
	\draw[cube] (-1,-1,2) -- (1,-1,0);
	
	\draw[cube] (-1,-1,2) node {$\bullet$};
	\draw[cube] (-1,-1,2) node[below] {$u_3$};
	
	\draw[cube] (-1,-1,0) node {$\bullet$};
	\draw[cube] (-1,-1,0) node[left] {$u_4$};
	
	\draw[cube] (1,-1,0) node {$\bullet$};
	\draw[cube] (1,-1,0) node[right] {$u_5$};
	
	\draw[cube] (-1,1,0) node {$\bullet$};
	\draw[cube] (-1,1,0) node[left] {$u_6$};
	
	\draw[cube] (0,-1,0) node {$\bullet$};
	\draw[cube] (.1,-1,0) node[above] {$u_9$};
	
	\draw[common] (0,0,0) node {$\bullet$};

	
	
\end{tikzpicture}
\caption{A depiction of the minimal generators $u_i := u_{\rho_i}$ of the cone $\sigma\subseteq \R^5$ when projected to the first three coordinates. Here, the minimal generators in black have their fourth coordinate $1$ and those in magenta have fifth coordinate $1$. The purple vertex is $(0,0,0)$ and the only point where the tetra intersect after projecting down to the first three coordinates. The cone $\sigma$ is a cone over the Cayley product of these two tetrahedra. }
\label{fig}
\end{figure}

Note that the cone $\sigma^\vee$ is completely split Gorenstein cone of index 2 with respect to $\mathfrak{n} = (0,0,0,1,1)$. We take $\Xi_W = H_{\mathfrak{n}}(1)\cap M \cap \sigma^\vee$, which  one can compute to be
$$
\Xi_W = \{(1,0,0,1,0), (0,1,0,1,0), (0,0,1,0,1), (-1,-1,-1,0,1), (0,0,0,1,0), (0,0,0,0,1)\}. 
$$
We name each of these above lattice points $m_1, \dots, m_6 \in \Xi_W$ (in the above order), and can write the following global function on any fan $\Sigma$ with $\Sigma(1) = \{ \rho_1, \dots, \rho_{12}\}$. 
\begin{align*}
W =  \sum_{i=1}^6 c_i \prod_{\rho\in \sigma(1)} x^{\langle m_i, u_{\rho_i}\rangle } &= c_1 x_1^2 x_5^2x_9x_{11} + c_2 x_2^2 x_6^2 x_{10}x_{11} + c_3 x_3^2 x_7^2 x_{10} x_{12} \ + \\ 
	&\qquad  c_4 x_4^2 x_8^2 x_9x_{12} + c_5 x_3x_4x_5x_6x_9 x_{11} + c_6 x_1x_2x_7x_8x_{10}x_{12}.
\end{align*}
We can use $W$ to define complete intersections in different toric varieties. Since $\sigma, \sigma_1, \sigma_2$ all are completely split Gorenstein cones of index 2, there exists fans $\Sigma, \Sigma_1,$ and $\Sigma_2$ that are total spaces of rank two vector bundles over dimension 3 toric varieties. 

For $\Sigma_1$, we star subdivide with respect to $\rho_{11}$ and $\rho_{12}$ as $u_{\rho_{11}}+u_{\rho_{12}}= \mathfrak{n}$. In this case, we can reduce the potential to
$$
W' = (c_1 x_1^2 + c_2 x_2^2 + c_5x_3x_4)x_{11} + (c_3 x_3^2 +  c_4 x_4^2 + c_6 x_1x_2)x_{12}.
$$
and one can compute that this corresponds to a complete intersection 
$$
\mathcal{Z}' := Z(c_1 x_1^2 + c_2 x_2^2 + c_5x_3x_4, c_3 x_3^2 +  c_4 x_4^2 + c_6 x_1x_2) \subseteq [\mathbb{P}^3 / (\Z/4\Z)],
$$
where a generator $g$ of the $\Z/4\Z$ acts on $\P^3$ by $g\cdot (x_1:x_2:x_3:x_4) = (x_1:-x_2:ix_3:-ix_4)$.

On the other hand, for $\Sigma_2$, we star subdivide with respect to $\rho_9$ and $\rho_{10}$ as $u_{\rho_{9}}+u_{\rho_{10}}= \mathfrak{n}$. The potential reduces to 
$$
W''= (c_1 x_1^2 +c_4 x_4^2+ c_5 x_3x_4)x_9 + (c_2 x_2^2  + c_3 x_3^2 + c_6 x_1x_2)x_{10} 
$$
and this corresponds to the complete intersection 
$$
\mathcal{Z}'' = Z(c_1 x_1^2 +c_4 x_4^2+ c_5 x_3x_4, c_2 x_2^2  + c_3 x_3^2 + c_6 x_1x_2) \subseteq [\mathbb{P}^3 / (\Z/2\Z)],
$$
where the generator $g$ of $\Z/2\Z$ acts on $\P^3$ by $g\cdot (x_1:x_2:x_3:x_4) = (-x_1:-x_2:x_3:x_4)$.

If the $c_i$ are generic, one can check that both $\mathcal{Z}'$ and $\mathcal{Z}''$ are smooth Calabi-Yau orbifolds. Lastly, there is a $\Sigma$ with support $\sigma$ that is a rank 2 vector bundle and one can use $W$ to define a Calabi-Yau complete intersection $[\mathcal{Z}/G]$ in it. By Corollary~\ref{Cor:csrG gives cat resn}, there is a derived equivalence between $[\mathcal{Z}/G]$ and both $\mathcal{Z}'$ and $\mathcal{Z}''$, hence $\dbcoh{\mathcal{Z}'} \cong \dbcoh{\mathcal{Z}''}$. 

Note the ``shuffling'' of monomials that happens between the two potentials $W'$ and $W''$, which happens when different rays are used as the rays corresponding to the bundle coordinates. This can be used to make interesting equivalences between CICY's. In this example, it can be shown that $\mathcal Z'$ and $\mathcal Z''$ are birational, but it is not clear if the derived equivalent Calabi-Yau complete intersections are always birational for higher-dimensional examples.

\subsection{A higher-dimensional generalization of the Libgober-Teitelbaum family}

 In \cite{LT93}, Libgober and Teitelbaum proposed a mirror to a highly symmetric Calabi-Yau complete intersection given by two cubics in $\P^5$. In \cite{Malter}, it was proven to be derived equivalent to the Batyrev-Borisov mirror to two cubics in $\P^5$. In this subsection, we look at the most natural generalization to the Libgober-Teitelbaum family, and show it has a crepant categorical resolution using the Batyrev-Borisov mirror to the complete intersection of two degree $n$ polynomials in $\P^{2n-1}$. We fix $n\in \Z$, $n\ge 2$, throughout the below.

Define two polynomials
\begin{align*}
Q_{1,\lambda}&=x_1^n+x_2^n+\dots+x_n^n-\lambda x_{n+1}\dots x_{2n}, \\
Q_{2,\lambda}&=x_{n+1}^n+x_{n+2}^n+\dots+x_{2n}^n-\lambda x_1\dots x_n.
\end{align*}
Their complete intersection $Z_{\lambda}:=Z(Q_{1,\lambda}, Q_{2,\lambda})\subseteq\P^{2n-1}$ is a smooth Calabi-Yau complete intersection in $n=3$ for $\lambda$ such that $\lambda^6\ne 0,3^6$ and it is a singular complete intersection otherwise. It is also highly symmetric.

Denote by $\zeta_n$ a primitive $n$-th root of unity. Consider $\alpha_i,\beta_i\in \Z\pmod{n}$ ($1\le i\le n-1$) and $\delta\in \Z\pmod{n^2}$ such that
\[
\zeta_{n^2}^\delta=\zeta_n^{\beta_1+\dots+\beta_{n-1}}=\zeta_n^{\alpha_1+\dots+\alpha_{n-1}}.
\]
Consider the following subgroup $G_n$ of $PGL(2n-1,\C)$, given by automorphisms of the form\[
g_{\underline{\alpha},\underline{\beta},\delta}=\left\{\begin{array}{ll}
  \operatorname{diag}\left( \zeta_{n^2}^\delta,\zeta_{n^2}^\delta\zeta_n^{\alpha_1},\dots,\zeta_{n^2}^\delta\zeta_n^{\alpha_{n-1}},\zeta_{n^2}^{-\delta}\zeta_n^{\beta_1},\dots,\zeta_{n^2}^{-\delta}\zeta_n^{\beta_{n-1}},\zeta_{n^2}^{-\delta}\right)  & \text{if } n \text{ odd,} \\
   \operatorname{diag}\left( \zeta_{2n^2}^\delta,\zeta_{2n^2}^\delta\zeta_n^{\alpha_1},\dots,\zeta_{2n^2}^\delta\zeta_n^{\alpha_{n-1}},\zeta_{2n^2}^{-\delta}\zeta_n^{\beta_1},\dots,\zeta_{2n^2}^{-\delta}\zeta_n^{\beta_{n-1}},\zeta_{2n^2}^{-\delta}\right)  & \text{if } n \text{ even.}
\end{array}\right.
\]
One can check that the action of $G_n$ on $\P^{2n-1}$ acts invariantly on the variety $Z_{\lambda}$, hence we can define the orbifold $[\mathcal{Z}_n/G_n]:=Z(Q_{1,n,\lambda},Q_{2,n,\lambda})\subseteq [\P^{2n-1}/G_n]$. For $n=3$, this is the Libgober-Teitelbaum mirror to the complete intersection of two cubics in $\P^5$, which is proven to be derived equivalent to members of the Batyrev-Borisov mirror family in \cite{Malter}. However, for $n\geq 4$, the techniques for proving derived equivalence to the corresponding Batyrev-Borisov construction used in loc. cit. do not work any further, as $[\mathcal{Z}_n/G_n]$ is singular.  Using the results of \textsection~\ref{sec: Cicy}, we will demonstrate instead that the natural result of applying the Batyrev-Borisov construction yields categorical resolutions to the family $[\mathcal{Z}_n/G_n]$.

We start by noting that since $[\mathcal{Z}_n/G_n]$ is a complete intersection of the two $Q_{i,\lambda}$ in $\mathcal{X}_n=[\P^{2n-1}/G_n]$, there is a corresponding gauged LG model with superpotential 
\[
W=u_1Q_{1,\lambda}+u_2Q_{2,\lambda},
\]
This will be a global function for the total space $\tot(\O_{\mathcal{X}_n}(-D_1)\oplus\O_{\mathcal{X}_n}(-D_2))$ of a rank two vector bundle where $Q_{i,\lambda} \in \Gamma(\mathcal{X}_n, \O_{\mathcal{X}_n} (D_i))$. One can construct the toric variety for this vector bundle, which will be $N_{\R}$ where $N=\Z^{2n+1}$. Give this lattice the standard $\Z$-basis $e_1, \dots, e_{2n+1}$. We now will define the rays $\rho_1, \dots, \rho_{2n}, \tau_1, \tau_2$ of the fan $\Sigma_{\mathcal{X}_n}$ by giving its primitive generators. To do so, we first write $\delta_1 := \sum_{i=1}^n e_i$ and $\delta_2 := \sum_{i=n+1}^{2n-1} e_i$. Then the primitive generators for the rays are
\begin{equation}\begin{aligned}\label{first rays for LT}
u_{\rho_i} &= ne_i - \delta_2 + e_{2n+1}, \text{ for $1 \le i \le n$};\\
u_{\rho_i} &= -\delta_1 + n e_i + e_{2n}, \text{ for $n+1 \le i \le 2n-1$};\\
u_{\rho_{2n}} &= -\delta_1 + e_{2n};\\
u_{\tau_1} &= e_{2n}; \\
u_{\tau_2} &= e_{2n+1}.
\end{aligned}\end{equation}
We associate the coordinates $x_1, \dots, x_{2n}$ to $\rho_1, \dots, \rho_{2n}$ and $u_1, u_2$ to $\tau_1, \tau_2$. One can see that $\tot(\O_{\mathcal{X}_n}(-D_1)\oplus\O_{\mathcal{X}_n}(-D_2))$ is the toric variety corresponding to the fan obtained by star subdividing the cone $\operatorname{Cone}(\rho_1, \dots, \rho_{2n}, \tau_1, \tau_2)$ along $\tau_1$ and $\tau_2$.

We now move to the potential $W$. Write $e_i^*\in M$ for the dual basis vector to $e_i$. Define  $\delta_1^* := \sum_{i=1}^n e_i^*$ and $\delta_2^* := \sum_{i=n+1}^{2n-1} e_i^*$.  Next, one can compute that $W = \sum_{i=1}^{2n} x^{m_i} -  \lambda x^{m_{2n+1}} - \lambda x^{m_{2m+2}}$, where 
\begin{equation}\begin{aligned}
m_i &= e_i^* + e_{2n}^*, \text{ for $1 \le i \le n$}, \\
m_i &= e_i^* + e_{2n+1}^*, \text{ for $n+1 \le i \le 2n-1$}, \\
m_{2n} &= (- \textstyle\sum_{i=1}^{2n-1} e_i^*) + e_{2n+1}^*, \\
m_{2n+1} &= e_{2n}^*, \\
m_{2n+2} &= e_{2n+1}^*. 
\end{aligned}\end{equation}
Hence $\Xi_W = \{m_1, \dots, m_{2n+2}\}$ and $\sigma_W = \operatorname{Cone}(\Xi_W)$. The dual cone $\sigma_W^\vee$ is then the cone over the $4n+2$ points $u_{\rho_1}, \dots, u_{\rho_{4n}}, u_{\tau_1}, u_{\tau_2}$, where $u_{\rho_1}, \dots, u_{\rho_{2n}}, u_{\tau_1}, u_{\tau_2}$ are as in~\eqref{first rays for LT} and 
\begin{equation}\begin{aligned}
u_{\rho_{2n+i}} &= - \delta_1 + ne_i + e_{2n}, \text{ for $1 \le i \le n$}, \\
u_{\rho_{3n+i}} &= - \delta_2 + ne_{n+i} + e_{2n+1}, \text{ for $1 \le i \le n-1$}, \\
u_{\rho_{4n}} &= -\delta_2 + e_{2n+1}. 
\end{aligned}\end{equation}
The cone $\sigma_W^\vee$ is completely split reflexive Gorenstein of index 2. 
Thus, Lemma \ref{Lem:sigmaW dual csrG gives Assumption} implies that $\sigma_W^\vee$ fulfills Assumption \ref{conical assumption}. 

Denote by $\Sigma'_{-D_1',-D_2'}$ the fan whose support is $\sigma_W^\vee$ in the Assumption $\ref{conical assumption}$. Write $x_i$ for the variables associated to the new rays $\rho_i$.  The extended superpotential $\bar W$ on the partial compactification takes the form
\begin{align*}
\bar W &=u_1(x_1^nx_{2n+1}^n+\dots+x_n^{n}x_{3n}^n-\lambda x_{n+1}\cdots x_{3n}) \ + \\ 
&\qquad u_2(x_{n+1}^nx_{3n+1}^n+\dots+x_{2n}^nx_{4n}^n-\lambda x_1\cdots x_n\cdot x_{3n+1}\cdots x_{4n}).
\end{align*}
Then one obtains a complete intersection $[\mathcal{Z}'_n / G_n'] \subseteq [U_{\Sigma'} / G_{\Sigma'}]$. 

Lastly, one can verify that the complete intersection $[\mathcal{Z}'_n/G_n']$ as in \eqref{eq: 2ndCI} is smooth when $\lambda^{2n} \ne 0,n^{2n}$, hence Corollary \ref{2 CICYs} implies that we have a categorical resolution of $[\mathcal{Z}_n/G_n]$ via $[\mathcal{Z}'_n/G_n']$.

\begin{remark}
    For $n=2$, the complete intersection $[\mathcal{Z}_2/G_2]\subseteq [\P^{3}/G_2]$ is smooth and Corollary \ref{2 CICYs} yields a derived equivalence between $[\mathcal{Z}_2/G_2]$ and $[\mathcal{Z}'_2/G_2']$.
    Furthermore, the $n=2$ complete intersection $[\mathcal{Z}_2/G_2]\subseteq[\P^3/G_2]$ features in \textsection\ref{exa:diffbund}.
\end{remark}

\subsection{Mirror Constructions}

Take a $\Q$-Fano toric variety corresponding to a polytope $\Delta$. 
In \cite{ACG}, Artebani, Comparin and Guilbot prove that general hypersurfaces associated to a special linear system corresponding to \emph{canonical} subpolytopes $\Delta'$ of the anticanonical polytope $\Delta$ are Calabi-Yau. In particular, given a $\Q$-Fano toric variety with anticanonical polytope $\Delta_2$, one can take a special linear system corresponding to a canonical polytope $\Delta_1$ and consider its corresponding family of Calabi-Yau varieties. Here $(\Delta_1, \Delta_2)$ form a \emph{good pair} if both $\Delta_1$ and $\Delta_2^*$ are canonical. Consequently, $(\Delta^*_2, \Delta^*_1)$ also form a good pair, forming a duality. This generalized both mirror constructions of Batyrev-Borisov and Berglund-H\"ubsch-Krawitz. In the former case, when $\Delta$ is reflexive, then $(\Delta, \Delta)$ is a good pair and one recovers Batyrev duality. There have been recent work by Rossi that generalizes this work and suggests an iterative process of doing this, by introducing what is known as $f$-duality \cite{Rossi20, Rossi23}. In this case, however, the type of singularities that can arise is unclear.

In \cite[Theorem 1.2]{DFK}, the authors proved that given good pairs $(\Delta_1, \Delta_2)$ and $(\Delta_1', \Delta_2)$ members of their dual families were derived equivalent. This is expected, as when one has multiple constructions for the mirror of two symplectomorphic manifolds, these ``multiple'' mirrors must be derived equivalent according to the Homological Mirror Symmetry Conjecture. As $\Delta_1$ and $\Delta_1'$ define two different families of hypersurfaces in the same toric variety, a generalization of Moser's theorem would imply a symplectomorphism.

One expects a similar situation in the complete intersection case. To our knowledge, there is no sufficient criteria for polytopes defining complete intersections that are Calabi-Yau that is weaker than using the standard nef partition of reflexive polytopes for codimension higher than one. Consider nef partitions of reflexive polytopes $\Delta_1 = \Delta_{1,1} + \dots + \Delta_{1,r}$, 
and $\Delta_2 = \Delta_{2,1} + \dots + \Delta_{2,r}$ in $M_{\R}$ where one has the inclusion of  Cayley polytopes 
\begin{align*}
\Delta_{1,1} * \cdots * \Delta_{1,r}  \subseteq \Delta_{2,1} * \dots * \Delta_{2,r}.
\end{align*}
Since $\{\Delta_{i,j}\}_{j=1}^r$ is a nef partition, it corresponds to vector bundles $\mathcal{V} = \bigoplus_{j=1}^r \O_{X_{\Sigma_{\Delta_i}}}(-D_{i,j})$ over the toric varieties $X_{\Sigma_{\Delta_i}}$, where $\Sigma_{\Delta_i}$ is the normal fan to $\Delta_i$ and $D_{i,j}$ is the divisor associated to the polytope $\Delta_{i,j}$. The lattice points in $\Delta_{1,1} * \cdots * \Delta_{1,r}$ correspond to global functions of the vector bundle. Each global function corresponds to a (stacky) complete intersection $\mathcal{Z}_i$. We then have by Corollary~\ref{Cor:csrG gives cat resn} a derived equivalence between $\mathcal{Z}_1$ and $\mathcal{Z}_2$. 

\begin{question}
Is there a combinatorial condition for Cayley products of length $r>1$ that generalizes canonical in $r=1$ where one obtains Calabi-Yau orbifolds? Does some higher codimension version of the mirror construction of Artebani, Comparin and Guilbot hold? 
\end{question}

If so, then, when one can use such a new mirror construction or Batyrev-Borisov, then we expect for there to be a derived equivalence between these new Calabi-Yau mirrors and those constructed by Batyrev-Borisov that can be proven using exoflops.

\bibliography{KellyMalter}
\bibliographystyle{amsalpha}
\end{document}